\documentclass[a4paper,reqno,10pt]{amsart}

\usepackage{amsmath}
\usepackage{amsthm}
\usepackage{amssymb}
\usepackage{amscd} 

\usepackage{mathabx}

\usepackage{bbm} 
\usepackage{mathrsfs} 

\usepackage[applemac]{inputenc} 

\usepackage{verbatim} 

\swapnumbers 

\newtheoremstyle{exercise} 
  {3pt} 
  {3pt} 
  {\scriptsize\rmfamily} 
  {
\parindent} 
  {\rmfamily\scshape} 
  {.} 
  {.5em} 
  {} 

\newtheoremstyle{newplain}
  {5pt}
  {5pt}
  {\itshape}
  {}
  {\rmfamily\scshape}
  {. ---}
  {.5em}
  {}

\newtheoremstyle{newremark}
  {5pt}
  {5pt}
  {\rmfamily}
  {}
  {\rmfamily\scshape}
  {. ---}
  {.5em}
  {}

\theoremstyle{newplain}
\newtheorem*{Theorem*}{Theorem} 

\theoremstyle{newplain}
\newtheorem{Theorem}{Theorem}
\newtheorem{Lemma}[Theorem]{Lemma}
\newtheorem{Corollary}[Theorem]{Corollary}
\newtheorem{Proposition}[Theorem]{Proposition}
\newtheorem{Conjecture}[Theorem]{Conjecture}
\newtheorem{Definition}[Theorem]{Definition}

\theoremstyle{newremark}
\newtheorem{Remark}[Theorem]{Remark}

\newtheorem{Claim}[Theorem]{Claim}

\theoremstyle{exercise}

\numberwithin{Theorem}{section}
\numberwithin{Exercise}{subsection}

\newcommand{\R}{\mathbb{R}} 
\newcommand{\Rm}{\R^m}

\newcommand{\ind}{\mathbbm{1}} 

\newcommand{\calC}{\mathscr{C}}
\newcommand{\calD}{\mathscr{D}}
\newcommand{\calE}{\mathscr{E}}

\newcommand{\calG}{\mathscr{G}}
\newcommand{\calH}{\mathscr{H}}

\newcommand{\calL}{\mathscr{L}}
\newcommand{\calM}{\mathscr{M}}

\newcommand{\calQ}{\mathscr{Q}}

\DeclareMathOperator{\rmHom}{\mathrm{Hom}} 

\DeclareMathOperator{\rmsupp}{\mathrm{supp}} 
\DeclareMathOperator{\rmtrace}{\mathrm{trace}} 

\newcommand{\lseg}{\boldsymbol{[}\!\boldsymbol{[}}
\newcommand{\rseg}{\boldsymbol{]}\!\boldsymbol{]}}

\def\XXint#1#2#3{{%
\setbox0=\hbox{$#1{#2#3}{\int}$}
\vcenter{\hbox{$#2#3$}}\kern-.5\wd0}}

\newcommand{\lno}{\left\bracevert}
\newcommand{\rno}{\right\bracevert}

\renewcommand{\leq}{\leqslant}
\renewcommand{\geq}{\geqslant}
\renewcommand{\subset}{\subseteq}

\newcommand{\QQl}{\mathscr{Q}_Q(\ell_2)}

\begin{document}
\title[Multiple valued maps]{Multiple valued maps into a separable Hilbert space that almost minimize their $p$ Dirichlet energy or are squeeze and squash stationary}

\begin{abstract}
Let $f : U\subset\Rm \to \calQ_Q(\ell_2)$ be of Sobolev class $W^{1,p}$, $1 < p < \infty$. If $f$  almost minimizes its $p$ Dirichlet energy then $f$ is H\"older continuous. If $p=2$ and $f$ is squeeze and squash stationary then $f$ is in VMO.
\end{abstract}

\author{Philippe Bouafia}
\address{
Equipe d\'analyse harmonique,
Universit\'e Paris-Sud,
B\'atiment 425,
91405 Orsay Cedex,
France
}
\email{philippe.bouafia@gmail.com}

\author{Thierry De Pauw}
\address{
Institut de Math\'ematiques de Jussieu,
Equipe G\'eom\'etrie et Dynamique,
B\^atiment Sophie Germain,
Case 7012,
75205 Paris Cedex 13,
France
}
\email{depauw@math.jussieu.fr}

\author{Changyou Wang}
\address{
Department of Mathematics,
University of Kentucky,
Lexington, KY 40506,
USA}
\email{cywang@ms.uky.edu}
\maketitle

\section{Introduction}
We let $\ell_2$ denote the usual infinite dimensional separable Hilbert space.
For any positive integer $Q$, let $\QQl$ denote the space of unordered $Q$-tuples of elements in $\ell_2$. Thus $\QQl$ is the quotient of $(\ell_2)^Q$ under the action of the symmetric group given by
$$
\sigma \cdot (u_1, \dots, u_Q) \mapsto (u_{\sigma(1)}, \dots, u_{\sigma(Q)}).
$$
The equivalence class of $(u_1, \dots, u_Q)$ will be denoted throughout $u = \oplus_{i=1}^Q \lseg u_i \rseg$.
The distance between two points $u = \oplus_{i=1}^Q \lseg u_i \rseg$ and $v=\oplus_{i=1}^Q \lseg v_i\rseg$ in $\QQl$ is defined by
\begin{equation}
\calG(u,v)=\min_{\sigma\in S_Q}\sqrt{\sum_{i=1}^Q \|u_i-v_{\sigma(i)}\|^2}.
\end{equation}
For the sake of simplicity, we will often use the notation
$$ |u| := \calG(u , Q\lseg 0 \rseg).$$

If $e_1, \dots, e_m$ is an orthonormal basis of $\R^m$, then
$$ \langle A, B \rangle := \sum_{i=1}^m \langle A(e_i), B(e_i) \rangle$$
defines a scalar product in $\mathrm{Hom}(\R^m, \ell^2)$, which is independent of the choice of $e_1, \dots, e_m$. The induced norm $\|A\|_{{\rm{HS}}} = \sqrt{\langle A, A \rangle}$ is the Hilbert-Schmidt norm of $A$.
Whenever $D = \oplus_{i=1}^Q \lseg D_i \rseg \in \calQ_Q(\mathrm{Hom}(\R^m, \ell_2))$ is a unordered $Q$-tuple of linear maps, we define
$$ \lno D \rno := \sqrt{\sum_{i=1}^Q \| D_i \|_{{\rm{HS}}}^2}. $$
We refer to \cite{BDPG} for the definition of multiple valued Sobolev space $W_p^1(U, \QQl)$, $U \subset \Rm$ open and $1 < p < \infty$. We merely mention that each $f \in W_p^1(U, \QQl)$ is approximately differentiable almost everywhere, its approximate differential $Df$ being itself a $Q$-valued map $\Rm \to \calQ_Q(\rmHom(\Rm,\ell_2))$. The $p$-energy of $f$ is
$$
\calE_p(f, U) := \left(\int_U \lno Df \rno^p \right)^{1/p} < \infty .
$$

Existence of minimizers of the $p$-energy for the Dirichlet problem with Lipschitz boundary data is established in \cite{BDPG}. In the present paper we prove their (interior) H\"older continuity. In fact we work in the more general setting of almost minimizers which we now recall. Let $\omega:[0,1]\to\mathbb R_+$ be a monotone increasing function with
$\displaystyle\omega(0)=\lim_{r\downarrow 0^+}\omega(r)=0$. Such a $\omega$ will be referred to as a {\em modular function}. For $1<p<+\infty$ and an open set
$U\subset\mathbb R^m$, we say that $u\in W^{1}_{p}(U, \QQl)$ is a $(\omega, p)$-Dir-minimizing function, if for any ball $B(x,r)\subset U$,
\begin{equation}\label{p-min}
\calE_p(u, B(x,r))\le (1+\omega(r))\calE_p(v, B(x,r)),
\end{equation}
whenever $v\in W^{1}_{p}(B(x,r),\QQl)$ and $\rmtrace u_{| B(x,r)} = \rmtrace v_{| B(x,r)}$. When $\omega=0$ we simply call $u$ a  $p$-Dir-minimizing function. Our first result is the following.

\begin{Theorem} \label{holder_reg} For any modular function $\omega$, there exists $\delta=\delta(p, m, \omega, Q)\in (0,1)$ such that
any $(\omega,p)$-Dir-minimizing function $u\in W^{1}_{p}(U, \QQl)$ is H\"older continuous in $U$
with an exponent $\delta$. Moreover, for any ball $B(x,2r)\subset U$,
\begin{equation}
\|u\|_{C^\delta(B(x,r))}\le C r^{p-m}\calE_p(u, B(x,2r)). \label{holder-est}
\end{equation}
\end{Theorem}
 The H\"older continuity of $2$-Dir-minimizing function into $\calQ_Q(\ell_2^n)$, i.e. when the target space is finite dimensional,  was first proved by F.J. Almgren
in his seminal work \cite{Almgren} and proved again by C. De Lellis and E. Spadaro \cite{DS} \cite{D} very recently.
In order to establish our result we define comparison maps by means of homogeneous extensions of the local boundary data (Lemma \ref{p-min0}) and an interpolation procedure (Theorem \ref{ext}) inspired by S. Luckhaus \cite{Lu}.

The Dir-minimizing property of $f$ leads to stationarity with respect to domain and range variations: We say $f$ is, respectively, {\em squeeze} and {\em squash stationary}. When $Q \geq 2$ the squeeze and squash stationarity of $f$ does not imply that it locally minimizes its energy. Thus the above regularity result does not apply to stationary maps. Here we assume $p=2$ and we contribute the VMO regularity of squeeze and squash stationary maps $f \in W^1_2(U;\calQ_Q(\ell_2))$, Proposition \ref{VMO}. We observe that the measure $\displaystyle\mu_f : A \mapsto \int_A \lno Df \rno^2$ is $m-2$ monotonic, i.e. that $r \mapsto r^{2-m} \mu_f(B(x,r))$ is nondecreasing, $x \in U$. We also notice, as other authors have, that the monotonicity of frequency, established by F.J. Almgren, depends solely upon the stationary property of $f$. This in turn shows that $\Theta^{m-2}(\mu_f,x)=0$ for all $x \in U$, which implies VMO via the Poincar\'e inequality. Furthermore $\displaystyle\lim_{r \to 0} r^{2-m} \mu_f(B(x,r))=0$ uniformly in $x \in U$ according to Dini's Theorem, which is in fact a kind of uniform VMO property. The continuity of $f$ would ensue from a sufficiently fast decay of
$\displaystyle\omega(r) = \sup_x r^{2-m} \mu_f(B(x,r))$. We establish the upper bound $\omega(r) \leq C | \log r |^{-\alpha}$ for some $0 < \alpha < 1$, which does not verify the suitable Dini growth condition.

\section{Radial comparison}

 The symbol $\oplus$ also denotes the concatenation operation
$$
\calQ_K(\ell_2) \times \calQ_L(\ell_2) \to \calQ_{K+L}(\ell_2).
$$
The barycenter of $u$ is
$$
\boldsymbol{\eta}(u) := \frac{1}{Q} \sum_{i=1}^Q u_i \in \ell_2
$$
and the translate of $u$ by $a \in \ell_2$ is
$$ \tau_a(u) := \bigoplus_{i=1}^Q \lseg u_i - a \rseg. $$

There are two crucial ingredients in the proof of the Theorem \ref{holder_reg}: a radial comparison lemma and an interpolation lemma.
This section is devoted to the radial comparison lemma. Let $B \subset \R^m$ denote the unit open ball. If $u \in W^1_p(B, \QQl)$, we let $\calC(p,Q,u_{|\partial B})$ be
\begin{equation} \calC(p, Q, u_{|\partial B}):=\inf\Big\{\calE_p(v, B):  v\in W^{1}_{p}(B, \QQl),  v= u \ {\rm{on}}\ \partial B\Big\}. \label{p-inf}
\end{equation}

\begin{Lemma} \label{p-min0} For any $M>0$, there exists $\eta_0 = \eta_0(p,m,Q,M)>0$  such that if
$u\in W^{1}_{p}(B, \QQl)$  satisfies
\begin{equation} \ {\rm{diam}}^p(\rmsupp \overline{u})\le M \calE_p(u,\partial B), \label{diameter}
\end{equation}
where $\overline{u}\in \QQl$ is a mean of $u$ on $\partial B$,
then we have
\begin{equation}
\calC(p, Q, u_{|\partial B}) \le \left(\frac{1}{m-p}-2\eta_0\right) \calE_p(u,\partial B), \label{p-inf1}
\end{equation}
In particular, there exists $\varepsilon_0 = \varepsilon_0(p,n,Q,M) > 0$ such that if $u\in W^{1}_{p}(B, \QQl)$ is an $(\omega,p)$-Dir-minimizing function and
$\omega(1)\le\varepsilon_0$, then we have
\begin{equation}\label{p-min1}
\calE_p(u, B)\le \Big(\frac{1}{m-p}-\eta_0\Big)  \calE_p(u,\partial B).
\end{equation}
\end{Lemma}
\begin{proof} After multiplying $u$ by a constant if necessary we may assume
$$ \calE_p(u,\partial B)=1.$$ Abbreviate $g = u_{|\partial B}$. For $\alpha>0$ to be chosen later, consider the radial competitor map
$$\forall x \in B, \quad v_\alpha(x)=\|x\|^\alpha g\left(\frac{x}{\|x\|}\right).$$
Since $v_\alpha=u$ on $\partial B_1$, it follows from the inequality (\ref{p-inf}) that
\begin{equation}\label{p-min2}
\calC(p, Q, u_{|\partial B})\le \calE_p(v_\alpha, B).
\end{equation}
Now we calculate $\calE_p(v_\alpha, B)$ as follows. Using the polar coordinates $(r,\theta) \in (0,1]\times\mathbb S^{m-1}$,
we have
$$\lno Dv_\alpha\rno=r^{\alpha-1}\Big\{\alpha^2|g(\theta)|^2+\lno D_{\mathbb S^{m-1}}g(\theta)\rno^2\Big\}^\frac12.$$
Hence
\setlength\arraycolsep{1.4pt}
\begin{eqnarray}
&&\calE_p(v_\alpha, B)\nonumber\\
&=&\int_0^1 r^{m-1+p(\alpha-1)}\int_{\partial B} \Big(\alpha^2|g(\theta)|^2+\lno D_{\mathbb S^{m-1}}g(\theta)\rno^2\Big)^{\frac{p}2}\,d\calH^{m-1}\theta\,dr\nonumber\\
&=&\frac{1}{m-p+p\alpha} \int_{\partial B}\Big(\alpha^2|g(\theta)|^2+\lno D_{\mathbb S^{m-1}}g(\theta)\rno^2\Big)^{\frac{p}2}\,d\calH^{m-1}\theta\,dr. \label{p-energy-v}
\end{eqnarray}
Next we want to estimate (\ref{p-energy-v}) from above. We distinguish between two cases.\\
\noindent{\it Case 1}: $p \in (1, 2]$.  Since $0<\frac{p}2\le1$, applying the elementary inequality
$$(a+b)^{\frac{p}2}\le a^{\frac{p}2}+b^{\frac{p}2}, \ a, \ b>0,$$ 
we have
\setlength\arraycolsep{1.4pt}
\begin{eqnarray*}
\big(\alpha^2|g(\theta)|^2+\lno D_{\mathbb S^{m-1}}g(\theta)\rno^2\big)^{\frac{p}2}
& \le & \alpha^p |g(\theta)|^p+\lno D_{\mathbb S^{m-1}}g(\theta)\rno^p,
\end{eqnarray*}
and hence
\setlength\arraycolsep{1.4pt}
\begin{eqnarray}
&&\calE_p(v_\alpha, B)\nonumber\\
&\leq &
\frac{1}{m-p+p\alpha} \big(\alpha^p\int_{\partial B} |g(\theta)|^p\,d\calH^{m-1}\theta
+\int_{\partial B} \lno D_{\mathbb S^{m-1}}g(\theta)\rno^p\,d\calH^{m-1}\theta\big) \nonumber\\
&\le & \frac{1}{m-p+p\alpha} \big(1+\alpha^p\int_{\partial B} |g(\theta)|^p\,d\calH^{m-1}\theta\big).
\end{eqnarray}

\noindent{\it Case 2}: $2<p\le n$. In this case, recall that the following inequality holds: there exists $C=C(p)>0$ such that
for any $\delta>0$,
$$\forall a,b >0, \quad (a+b)^{\frac{p}2}\le (1+\delta) a^{\frac{p}2} +C\delta^{-(\frac{p}2-1)}b^{\frac{p}2}.$$
Applying this inequality, we have that for any $0<\delta<1$,
\setlength\arraycolsep{1.4pt}
\begin{eqnarray}
&&\calE_p(v_\alpha, B)\nonumber\\
& \leq & \frac{C\delta^{-(\frac{p}2-1)}\alpha^p}{m-p+p\alpha} \int_{\partial B} |g(\theta)|^p\,d\calH^{m-1}\theta \nonumber \\
& & + \frac{1+\delta}{m-p+p\alpha}\int_{\partial B} \lno D_{\mathbb S^{m-1}}g(\theta)\rno^p\,d\calH^{m-1}\theta \nonumber\\
&\leq & \frac{1}{m-p+p\alpha} \big(C\delta^{-(\frac{p}2-1)}\alpha^p\int_{\partial B} |g(\theta)|^p\,d\calH^{m-1}\theta
+(1+\delta)\big).
\end{eqnarray}
To estimate $\int_{\partial B}|g(\theta)|^p\,d\calH^{m-1}\theta$, we argue as in \cite{DS} page 39. Write $\overline{u} =: \oplus_{i=1}^Q \lseg \overline{u}_i \rseg$. Let
$\hat{u}=\tau_{\boldsymbol{\eta}(\overline{u})}(u)$ and $\hat{g}=\tau_{\boldsymbol{\eta}(\overline{u})}(g)$ denote the translations of $u$ and $g$ by $\boldsymbol{\eta}(\overline{u})$. It is easy to see that
$\displaystyle\hat{g}=\hat{u}|_{\partial B}$.
It is clear that $\overline{\hat{g}} := \tau_{\boldsymbol{\eta}(\overline{u})}(\overline{u})$ is a mean of $\hat {g}$, and
$$|\overline{\hat g}|^2=\sum_{i=1}^Q \|\overline{u}_i - \boldsymbol{\eta}(\overline{u})\|^2\le Q {\rm{diam}}^2\rmsupp \overline{u} \le Q M^2.$$
By the Poincar\'e inequality, we have
\setlength\arraycolsep{1.4pt}
\begin{eqnarray}
\int_{\partial B}|\hat {g}|^pd\calH^{m-1} & \le &
2^{p}\big(\int_{\partial B} \calG({\hat g}, \overline{\hat g})^pd\calH^{m-1}
+\int_{\partial B}|\overline{\hat g}|^pd\calH^{m-1}\big) \nonumber \\
& \le &  C\left(\calE_p(\hat{u},\partial B)+M^p\right) \le C(1+M^p).
\end{eqnarray}

Since $\calC(p,Q, u_{|\partial B})$ is invariant under translations of $u_{|\partial B}$
we conclude that
\setlength\arraycolsep{1.4pt}
\begin{eqnarray*}
\calC(p,Q, u_{|\partial B}) &=&
\calC(p,Q, \tau_{\boldsymbol{\eta}(\overline{u})}(u)_{|\partial B})\nonumber\\
& \leq &
\calE_p(\tau_{\boldsymbol{\eta}(\overline{u})}(v_\alpha), B)\nonumber\\
&\leq &
\calM(m,p,M, \alpha, \delta)
\end{eqnarray*}
where
\begin{equation}\label{m-estimate}
\calM(m,p,M,\alpha,\delta) :=
\begin{cases}\frac{1}{m-p+p\alpha}\left(1+C(1+M^p)\alpha^p\right) & 1<p\le 2\\
\frac{1}{m-p+p\alpha}\left((1+\delta)+C\delta^{-(p/2-1)}(1+M^p)\alpha^p\right) & 2<p\le n.
\end{cases}
\end{equation}
Now we need to show the following claim: \\
\noindent{\it Claim}. There exist $\alpha_0>0$ and $\delta_0>0$ depending only on $p, m, M$ such that
\begin{equation}\label{gap2}
\calM(m,p, M, \alpha_0,\delta_0)<\frac{1}{m-p}.
\end{equation}
To establish (\ref{gap2}), we first consider the case $1<p\le 2$. By the definition, we have
$$
 \calM(m,p,M,\alpha,\delta)=\frac{1}{m-p+p\alpha}\left(1+C(1+M^p)\alpha^p\right)
$$
is independent of $\delta$.
It is readily seen that
\begin{equation}\label{gap_cond1}
\calM(m,p,M,\alpha,\delta)<\frac{1}{m-p}
\Leftrightarrow C(1+M^p) \alpha^p<\frac{p}{m-p}\alpha.
\end{equation}
Since $p > 1$ it is most obvious that \eqref{gap_cond1} holds provided $0 < \alpha \leq \alpha_0(m,p,M)$ is small enough.
Next we consider the case $2<p\le n$. In this case, we have
$$
 \calM(m,p,M,\alpha,\delta)=\frac{1}{m-p+p\alpha}\left((1+\delta)+C\delta^{-(\frac{p}2-1)}(1+M^p)\alpha^p\right).
$$
Again it is easy to see that
\begin{equation}\label{gap_cond2}
 \calM(m,p,M,\alpha,\delta)<\frac{1}{m-p}
\Leftrightarrow \delta+C\delta^{-(\frac{p}2-1)}(1+M^p)\alpha^p<\frac{p}{m-p}\alpha.
\end{equation}
Letting $\delta = \alpha^2$ the left member of \eqref{gap_cond2} becomes a constant multiple of $\alpha^2$. Thus the inequality is verified provided $0 < \alpha \leq \alpha_0(m,p,M)$ is small enough.
Combining together both cases, we see that there exists $\eta_0>0$ depending only on $p, m, Q, M$
such that (\ref{p-inf1}) holds.  To show (\ref{p-min1}), first observe that the $(\omega,p)$-Dir-minimality of $u$ and (\ref{p-inf1})
imply
\setlength\arraycolsep{1.4pt}
\begin{eqnarray*}
\calE_p(u, B) &\leq &  (1+\omega(1)) \calC(p, Q, u_{|\partial B}) \\
 & \leq & (1+\varepsilon_0)\big(\frac{1}{m-p}-2\eta_0\big)\\
&\leq & \frac{1}{m-p}-\eta_0+\big(\frac{\varepsilon_0}{m-p}-\eta_0\big) \\
& \le & \frac{1}{m-p}-\eta_0,
\end{eqnarray*}
provided $\varepsilon_0\le (m-p)\eta_0$.  Hence the proof is complete.  \end{proof}

An immediate consequence of the radial comparison lemma is the H\"older continuity of $(\omega,p)$-Dir-minimizing functions to $\ell_2$, which
implies Theorem \ref{holder_reg} holds for the case $Q=1$. More precisely, we have

\begin{Corollary} \label{holder_reg0} There exists $\eta_0>0$ depending only on $m, p$ such that
for any $u\in W^{1}_{p}(\partial B(0,r), \ell_2)$, one has
\begin{equation}\label{p-inf0}
\calC(p, 1, u_{|\partial B(0,r)})\le \big(\frac{1}{m-p}-2\eta_0\big) r\calE_p(u, \partial B(0,r)).
\end{equation}
\end{Corollary}
\begin{proof} By scaling, it suffices to show (\ref{p-inf0}) for $r=1$.  This follows from Lemma \ref{p-min0}. In fact, for $Q=1$, one has
that the diameter of $\rmsupp \overline{u}$ equals to $0$, \emph{i.e.} $M=0$ in the condition (\ref{diameter}). Hence $\calM(m, p, M, \alpha,\delta)
=\calM(m, p, \alpha,\delta)$, given by (\ref{m-estimate}). We define $\alpha_0, \delta_0, \eta_0$ as in \eqref{gap2} when $M$ is replaced by 0. Clearly, $\alpha_0$ and $\delta_0$ depend only on $p, m$ and
$$\calM(m, p, \alpha_0,\delta_0)\le \frac{1}{m-p}-2\eta_0.$$
This immediately implies (\ref{p-inf0}).
\end{proof}

\begin{Corollary} \label{holder_reg1} For any given modular function $\omega$, there exists $\delta=\delta(p,\omega,m)\in (0,1)$ such that
any $(\omega,p)$-Dir-minimizing function $u\in W^{1}_{p}(U, \ell_2)$ is H\"older continuous in $U$
with an exponent $\delta$. Moreover, for any ball $B(x,2r)\subset U$,
we have
\begin{equation}
\|u\|_{C^\delta(B(x,r))}\le C r^{p-m}\calE_p(u, B(x,2r)). \label{holder-est1}
\end{equation}
\end{Corollary}
\begin{proof}
Since $u\in W^{1}_{p}(U, \ell_2)$ is $(\omega,p)$-Dir-minimizing in $U$,  by (\ref{p-inf0}) we have that for any ball $B(x,r)\subset U$,
\setlength\arraycolsep{1.4pt}
\begin{eqnarray}\label{p-min13}
\calE_p(u, B(x,r))&\le& (1+\omega(r))\calC(p,1, u_{|\partial B(x,r)})\nonumber\\
&\le& (1+\omega(r))\big(\frac{1}{m-p}-2\eta_0\big) r\calE_p(u, \partial B(x,r)).
\end{eqnarray}
Since $\lim_{r\downarrow 0}\omega(r)=0$, there exists $r_0>0$ such that
$$\forall r \in (0, r_0], \quad (1+\omega(r))\big(\frac{1}{m-p}-2\eta_0\big)\le \frac{1}{m-p}-\eta_0.$$
Thus we have that
\begin{equation}\label{p-min12prim}
\calE_p(u, B(x,r))\le \big(\frac{1}{m-p}-\eta_0\big) r \calE_p(u, \partial B(x,r))
\end{equation}
holds for all $B(x,r)\subset U$ with $0<r\le r_0$. It is standard that integrating (\ref{p-min12prim}) over $r$ yields
that
\begin{equation}\label{p-decay}
\frac{1}{r^{m-p+\eta_0}}\calE_p(u, B(x,r))\le \frac{1}{r_0^{m-p+\eta_0}}\calE_p(u, B(x,r_0)),
\end{equation}
for all balls $B(x,r_0) \subset U$ and $r \in (0,r_0]$. This, combined with the Morrey decay lemma for $\ell_2$-valued functions, implies that
$u\in C^{\eta_0/p}(U)$ and
$$\|u\|_{C^{\eta_0/p}(B(x,r)) } \le C\frac{1}{r_0^{m-p+\eta_0}}\calE_p(u, B(x,r_0))$$
holds for $B(x, r_0)\subset U$ and $0<r\le r_0$. This completes the proof.
\end{proof}

\section{Interpolation}

In this section, we will establish an interpolation lemma. Such an interpolation property has been established
in $\calQ_Q(\ell_2^n)$ by F. Almgren \cite{Almgren}. However, the original proof by \cite{Almgren} is of extrinsic nature, \emph{i.e.} it
depends on the existence of a Lipschitz embedding of $\calQ_Q(\ell_2^n)$ into $\ell_2^N$ for some large positive integer $N=N(m, n, Q)$.
There seems to be no useful ersatz of this embedding in the case of $\QQl$.

\begin{Theorem} \label{ext}
For any $1<p\le m$ and $\varepsilon>0$, there exists $C=C(m, p, Q)>0$ such that if $g_1,g_2\in W^{1}_{p}(\partial B, \QQl)$, then there
exists $h\in W^{1}_{p}(B\setminus B(0,1-\varepsilon), \QQl)$ such that
\begin{equation}
\forall x \in \partial B, \quad h(x)=g_1(x), \qquad \forall x \in \partial B(0,1-\varepsilon), \quad h(x)=g_2\big(\frac{x}{1-\varepsilon}\big),
\label{ext1}
\end{equation}
and
\begin{equation}
\calE_p(h, B\setminus B(0,1-\varepsilon))
\le C\big( \varepsilon \sum_{i=1}^2\calE_p(g_i, \partial B)
+\varepsilon^{1-p}\int_{\partial B} \calG^p(g_1, g_2)\,d\calH^{m-1}\big). \label{hybrid_ineq1}
\end{equation}
\end{Theorem}

For $1<p<+\infty$, set
\begin{equation} \label{dim}
m_p=\begin{cases} p-1 & \ {\rm{if}}\ p\in \mathbb Z_+\\
\lfloor p\rfloor & \ {\rm{if}}\ p\notin \mathbb Z_+,
\end{cases}
\end{equation}
where $\lfloor p\rfloor$ denotes the integer part of $p$.

Here we provide an intrinsic proof of Theorem \ref{ext}, analogous to that by S. Luckhaus \cite{Lu}. The rough idea  is first to find
a suitable triangulation of $\partial B_1$ and then do interpolations up to $m_p$-dimensional skeletons by first suitably
approximating $g_1, g_2$ by Lipschitz maps  and  perform suitable
Lipschitz extensions from $0$-dimensional skeletons for all $m_p$-dimensional skeletons,
here we need an important compactness theorem similar to Kolmogorov's theorem in our context.
Finally we perform homogenous of degree zero extensions in skeletons of dimensions higher than $m_p$.
We denote the unit interval by $I := [-1,1]$.

For this, we need to establish the following lemma.
\begin{Lemma} \label{1d-ext} For any $1<p<\infty$ and $\varepsilon>0$, assume $m\le m_p$. There exists a constant $C=C(p,Q)>0$ such that if $g_1,g_2\in W^{1}_{p}(I^m, \QQl)$, then there is a map
$h\in W^{1}_{p}(I^m\times [-\varepsilon,\varepsilon], \QQl)$ such that
\begin{equation}
\begin{cases}
h(x, \ \varepsilon)=g_1(x) & x\in I^m,\\
h(x, -\varepsilon)=g_2(x) & x\in I^m, \label{ext2}
\end{cases}
\end{equation}
and
\begin{equation}
\calE_p(h, I^m \times [-\varepsilon, \varepsilon])
\le C\big( \varepsilon \sum_{i=1}^2\calE_p(g_i, I^m)
+\varepsilon^{1-p}\int_{I^m} \calG^p(g_1, g_2)\,d\calH^{m}\big). \label{hybrid_ineq2}
\end{equation}
\end{Lemma}

\begin{proof}  We adapt some notations from \cite{DS} page 62. Let us introduce $I_k := [-1-\frac1{k},1+\frac1{k}]$. For sufficiently large $k\in \mathbb Z_+$, decompose
$I_k^m$ into the union of $(k+1)^m$ cubes $\{C_{k,l}\}$, $1\le l\le (k+1)^m$,
with disjoint interiors, side length equal to $2/k$ and faces parallel to the coordinate hyperplane. Let $x_{k,l}$ denote their
centers so that
$$C_{k,l}=x_{k,l}+\big[-\frac{1}{k}, \frac1{k}\big]^m, \qquad 1\le l\le (k+1)^m.$$

We also decompose $I^m$ into the union of $k^m$ cubes $\{D_{k,l}\}$, $1\le l\le k^m$ and of side length $2/k$. Note that the centers of cubes in the collection $\{C_{k,l} : 1 \leq l \leq (k+1)^m\}$ are precisely the vertices of cubes in the collection $\{D_{k,l} : 1 \leq l \leq k^m\}$. Now we define two functions on the set of vertices
$h_1^k, h_2^k: \{x_{k,1},\dots, x_{k, (k+1)^m}\} \to \QQl$ by letting
$$h_i^k(x_{k, l})={\rm{a\ mean\ of }}\  g_i \ {\rm{on}}\ \big(\{x_{k,l}\}+\big[-\frac{2}{k}, \frac{2}{k}\big]^m\big)\cap I^m,
\ 1\le l\le (k+1)^m,$$
for $i=1,2$. Now we want to extend both $h_1^k$ and $h^k_2$ from the set of vertices to the cube $I^m$. For each cube
$D_{k,l}$, we let $V_{k,l}$ denote the set of vertices of $D_{k,l}$, consisting of $2^m$ points extracted from $\{x_{k,l'}\}_{1\le l'\le (k+1)^m}$,
and let $F_{k,l}^j$ denote the set of all faces of $D_{k,l}$ of dimension $j$ for $1\le j\le m-1$.
On the first cube
$D_{k,1}$ we claim that there exist Lipschitz functions
$h_1^{k,1}, h^{k,1}_2: D_{k,1}\to \QQl$ that are extensions of $h_{1|V_{k,1}}^k, h^k_{2|V_{k,1}}: V_{k,1}\to \QQl$, respectively,
such that for $i=1, 2$,
\begin{equation}
{\rm{Lip}}\big(h^{k,1}_i, F)\le C{\rm{Lip}}\big(h^{k}_i,  V_{k,1}\cap F),
\ \forall F\in F_{k,1}^j, 1\le j\le m. \label{lip_bound1}
\end{equation}
In particular, for $j=m$, (\ref{lip_bound1}) yields
\begin{equation}
 {\rm{Lip}}\big(h^{k,1}_i, D_{k,1}\big)\le C {\rm{Lip}}\big(h_i^k, V_{k,1}\big). \label{lip_bound2}
\end{equation}
Indeed, we apply finitely many times Theorem 2.4.3 of \cite{BDPG}: first extend the maps $h^k_{i|V_{k,1}}$ to \emph{each} edge in $F^1_{k,1}$ (thus apply Theorem 2.4.3 $\mathrm{Card}\, F^1_{k,1}$ times), then to \emph{each} $j$-dimensional face in $F_{k,1}^j$, for $j = 2, \dots, m$, by induction on $j$.

On all those cubes $D_{k,l}$ that are adjacent to $D_{k,1}$ (i.e., share a common ($m-1$)-dimensional face $\partial D_{k,l}\cap \partial D_{k,1}$
with $D_{k,1})$, we proceed similarly to find Lipschitz
functions $h_1^{k,l}, h_2^{k,l}: D_{k,l}\to \QQl$ that are Lipschitz extensions of $\widetilde{h_1^{k,l}}, \widetilde{h_2^{k,l}}:
(\partial D_{k,l}\cap \partial D_{k,1})\cup V_{k,l}\to \QQl$ respectively, where
$$\widetilde{h^{k,l}_i}(x) =
\left\{\begin{array}{ll}
h^{k,1}_i(x) & {\rm{if}}\ x \in \partial D_{k,l}\cap \partial D_{k,1} \\
h^k_i(x) & {\rm{if}} \ x \in V_{k,l}
\end{array}\right.$$
for $i=1,2$. Moreover, for $i=1, 2$ the following estimates hold:
\begin{equation}
{\rm{Lip}}(h^{k,l}_i, F)\le C{\rm{Lip}}(h^{k}_i,  V_{k,l}\cap F),
\ \forall F\in F_{k,l}^j, \ 1\le j\le m,  \label{lip_bound3}
\end{equation}
In particular, for $j=m$, (\ref{lip_bound3}) yields
\begin{equation}\label{lip_bound4}
{\rm{Lip}}\big(h^{k,l}_i, D_{k,l}) \leq C {\rm{Lip}}(h^k_i, V_{k,l}).
\end{equation}
Repeating the above procedure with all subcubes $D_{k,l}$ for $1\le l\le k^m$,
we will eventually obtain two Lipschitz functions ${\bf h}^k_1, {\bf h}^k_2: I^m \to \QQl$
such that
\begin{equation}\label{lip_bound5}
{\bf h}^k_1(x_{k,l})=h^k_1(x_{k,l});  \ {\bf h}^k_2(x_{k,l})=h^k_2(x_{k,l}), \ \forall 1\le l\le (k+1)^m,
\end{equation}
 and for $i=1, 2$,
\begin{equation}
{\rm{Lip}}({\bf h}^k_i, D_{k,l})\le C{\rm{Lip}}(h^k_i, V_{k,l}), \quad \forall 1\le l\le k^m.
\end{equation}
Now we want to find a Lipschitz map $h^k:I^m\times [-\varepsilon, \varepsilon]\to \QQl$ that is a suitable extension of
$$
(x,\varepsilon) \in I^m \times \{\varepsilon\} \mapsto {\bf h}_1^k(x), \quad (x, -\varepsilon) \in I^m \times \{-\varepsilon\} \mapsto {\bf h}_2^k(x).
$$ This can be done as follows.
For each cube $D_{k,l}$, $1\le l\le k^m$, we let $h^{k,l}: D_{k,l}\times [-\varepsilon, \varepsilon]\to \QQl$ be a Lipschitz extension
of
$$
(x,\varepsilon) \in D_{k,l} \times \{\varepsilon\} \mapsto {\bf h}_1^k(x), \quad (x, -\varepsilon) \in D_{k,l} \times \{-\varepsilon\} \mapsto {\bf h}_2^k(x)
$$
 such that
 \begin{itemize}
 \item
if two cubes $D_{k,l}$ and $D_{k,l'}$ ($l\neq l'$) share a common $(m-1)$-dimensional face $\partial D_{k,l}\cap \partial D_{k,l'}$, then $h^{k,l}$ and
$h^{k,l'}$ take the same value on the common $m$-dimensional face $(\partial D_{k,l}\cap \partial D_{k,l'}) \times [-\varepsilon, \varepsilon]$.
\item the following inequalities hold:
\setlength\arraycolsep{1.4pt}
\begin{eqnarray}
{\rm{Lip}}(h^{k,l}, \partial D_{k,l}\times [-\varepsilon, \varepsilon])
& \leq &
C\left({\rm{Lip}}({\bf h}_1^k, \partial D_{k,l})+{\rm{Lip}}({\bf h}_2^k, \partial D_{k,l})\right) \nonumber \\
& & +C\varepsilon^{-1}\sum_{x\in V_{k,l}}\calG({\bf h}_1^k(x), {\bf h}_2^k(x))\nonumber\\
& \leq &
C\left({\rm{Lip}}\big(h^k_1, V_{k,l})+{\rm{Lip}}(h^k_2, V_{k,l})\right) \nonumber \\
& & +C \varepsilon^{-1}\sum_{x\in V_{k,l} }
\calG(h_1^k(x), h_2^k(x)). \label{lip_bound7}
\end{eqnarray}
and
\setlength\arraycolsep{1.4pt}
\begin{eqnarray}
\label{lip_bound8}
{\rm{Lip}}(h^{k,l}, D_{k,l}\times [-\varepsilon,\varepsilon] ) & \le & C{\rm{Lip}}(h^{k,l}, \partial(D_{k,l}\times [-\varepsilon, \varepsilon]))\nonumber\\
&\le &
C\left({\rm{Lip}}({\bf h}_1^{k}, D_{k,l})+{\rm{Lip}}({\bf h}_2^{k}, D_{k,l})\right) \nonumber \\
& & +C {\rm{Lip}}(h^{k,l}, \partial D_{k,l}\times [-\varepsilon,\varepsilon])\nonumber\\
&\le &
C\left({\rm{Lip}}(h_1^{k}, V_{k,l})+{\rm{Lip}}(h_2^{k}, V_{k,l})\right) \nonumber \\
& & +C\varepsilon^{-1}
\sum_{x\in V_{k,l}}\calG(h_1^k(x), h_2^k(x)).
\end{eqnarray}
\end{itemize}
Finally we define $h^k:I^m\times [-\varepsilon, \varepsilon]\to \QQl$ by simply letting
$$h^k_{|D_{k,l}\times [-\varepsilon, \varepsilon]}=h^{k,l}, \quad \forall\ 1\le l\le k^m.$$
Obviously $h^k$ satisfies that $h^k(x,\varepsilon)={\bf h}^k_1(x)$ and
$h^k(x,-\varepsilon)={\bf h}^k_2(x)$ for $x \in I^m$.

We want to estimate the terms in the right hand side of (\ref{lip_bound8}).
It is easy to see that for any $1\le l\le k^m$,
$${\rm{Lip}}(h_1^k, V_{k,l}) \le C \max\left\{k \calG(h_1^k(x), h_1^k(x')) : \ x, x'\in V_{k,l} \ {\rm{are\ two\ adjacent\ vertices}}\right\}.$$
On the other hand, for two adjacent vertices $x, x'\in V_{k,l}$, by the definition of $h_1^k(x)$ and $h^k_1(x')$ and Poincar\'e's inequality we have
\setlength\arraycolsep{1.4pt}
\begin{eqnarray*}
&&\calG(h_1^k(x), h_1^k(x'))^p\\
& \le &
Ck^m\int_{((\{x\}+[-2k^{-1},2k^{-1}]^m)\cap (\{x'\}+[-2k^{-1},2k^{-1}]^m))\cap I^m}
 \calG(h_1^k(x),h_1^k(x'))^p\\
&\le &
Ck^m\int_{(\{x\}+[-2k^{-1}, 2k^{-1}]^m)\cap I^m}\calG(g_1(y), h_1^k(x))^p dy \\
& & +Ck^m\int_{(\{x'\}+[-2k^{-1},2k^{-1}]^m)\cap I^m}\calG(g_1(y), h_1^k(x'))^p dy\\
& \leq &  Ck^{m-p} \calE_p(g_1, (\{x\}+[-2k^{-1}, 2k^{-1}]^m)\cap I^m) \\
& & +Ck^{m-p}\calE_p(g_1, (\{x'\}+[-2k^{-1}, 2k^{-1}]^m)\cap I^m)\\
& \leq &  Ck^{m-p} \calE_p(g_1, \widetilde{D_{k,l}}\cap I^m),
\end{eqnarray*}
where $\widetilde{D_{k,l}}$ denotes cube:
$$\widetilde{D_{k,l}}=\{x_{k,l}\}+\big[-\frac3{k}, \frac3{k}\big]^m.$$
Thus we have
\begin{equation}\label{lip_bound9}
{\rm{Lip}}^p(h_1^k, V_{k,l})\le Ck^m \calE_p(g_1, \widetilde{D_{k,l}}\cap I^m).
\end{equation}
Similarly,  we have
\begin{equation}\label{lip_bound10}
{\rm{Lip}}^p(h_2^k, V_{k,l})\le Ck^m \calE_p(g_2, \widetilde{D_{k,l}}\cap I^m).
\end{equation}
While for $x\in V_{k,l}$, we have
\setlength\arraycolsep{1.4pt}
\begin{eqnarray}
\calG(h_1^k(x), h^k_2(x))^p & \leq & C\max_{y\in D_{k,l}} \left(\calG({\bf h}^k_1(y), h_1^{k}(x))^p +\calG({\bf h}^k_2(y), h_2^{k}(x))^p \right) \nonumber \\
& & +Ck^m\int_{D_{k,l}}\calG({\bf h}_1^{k}(y), {\bf h}_2^{k}(y))^pdy\nonumber\\
&\le& C\left(k^{-p}{\rm{Lip}}^p({\bf h}_1^{k}, D_{k,l})+k^{-p}{\rm{Lip}}^p({\bf h}_2^{k}, D_{k,l})\right] \nonumber \\
& & +Ck^m\int_{D_{k,l}}\calG({\bf h}_1^{k}(y), {\bf h}_2^{k}(y))^pdy\nonumber\\
&\leq & Ck^{m-p}\big(\calE_p(g_1, \widetilde{D_{k,l}}\cap I^m)
+\calE_p(g_2, \widetilde{D_{k,l}}\cap I^m) \big) \nonumber \\
& & +Ck^m \int_{D_{k,l}}\calG({\bf h}_1^{k}(y), {\bf h}_2^{k}(y))^pdy.
\label{lip_bound11}
\end{eqnarray}
With all these estimates, we can bound $\displaystyle\calE_p\big(h^k, I^m\times [-\varepsilon, \varepsilon]\big)$ as follows:
\setlength\arraycolsep{1.4pt}
\begin{eqnarray}\label{p-energy-est}
&&\calE_p(h^k, I^m\times [-\varepsilon, \varepsilon])\nonumber\\
 & = & \sum_{l=1}^{k^m}\calE_p(h^{k,l}, D_{k,l}\times [-\varepsilon, \varepsilon]) \nonumber\\
& \le & \frac{C\varepsilon}{k^m} \sum_{l=1}^{k^m} {\rm{Lip}}^p(h^{k,l}, D_{k,l}\times [-\varepsilon, \varepsilon])\nonumber\\
&\le &
C\varepsilon(1+\varepsilon^{-p}k^{-p}) \sum_{l=1}^{k^m}\big(\calE_p(g_1, \widetilde{D_{k,l}}\cap I^m)+\calE_p(g_2, \widetilde{D_{k,l}}\cap B)\big)\nonumber\\
& & +C\varepsilon^{1-p}\sum_{l=1}^{k^m}\int_{D_{k,l}}\calG({\bf h}_1^{k}(y), {\bf h}_2^{k}(y))^pdy\nonumber\\
&\leq &
C\varepsilon(1+(k\varepsilon)^{-p})\big(\calE_p(g_1, I^m)+\calE_p(g_2, I^m)\big) \nonumber \\
& & + C\varepsilon^{1-p}\int_{I^m}\calG({\bf h}_1^{k}(y),
{\bf h}_2^{k}(y))^pdy.
\end{eqnarray}
Observe that we have for $i=1,2$,
\setlength\arraycolsep{1.4pt}
\begin{eqnarray*}
\int_{I^m} \calG({\bf h}^k_i(y), g_i(y))^pdy & \le &
C\sum_{l=1}^{k^m} \int_{D_{k,l}} \left(\calG({\bf h}^k_i, h^k_i(x_{k,l}))^p + \calG(h^k_i(x_{k,l}), g_i)^p\right) \\
& \leq & Ck^{-p} \sum_{l=1}^{k^m} \calE_p(g_i, \widetilde{D_{k,l}} \cap I^m)  \\
& & + C \sum_{l=1}^{k^m} \int_{(\{x\} + [-2k^{-1}, 2k^{-1}]) \cap I^m} \calG(h^k_i(x_{k,l}), g_i)^p \\
& \leq & Ck^{-p} \sum_{l=1}^{k^m} \calE_p(g_i, \widetilde{D_{k,l}} \cap I^m)  \\
& \leq & Ck^{-p} \calE(g_i, I^m),
\end{eqnarray*}
which converges to 0 as $k$ goes to $\infty$. We would establish (\ref{hybrid_ineq2}) if we can show that there exists $h\in W^{1}_{p}(I^m\times [-\varepsilon,\varepsilon], \QQl)$
such that after passing to possible subsequences, $h^k\rightarrow h$ in $L^p(I^m\times [-\varepsilon, \varepsilon])$.

To see this, since $p>m$, it follows from Sobolev's embedding theorem that $g_i\in C^{1-m/p}(I^m, \QQl)$ for $i=1, 2$.
If we define
$$\calC_i=\Big\{y\in \ell_2\ : \ y\in {\rm{supp}}(g_i(x)) \ {\rm{for\ some}}\ x\in I^m\Big\}, \ i=1, 2,$$
then $\calC_i\subset \ell_2$ is a compact set for $l=1,2$. Hence
$$\calC:=\calC_1\cup \calC_2$$
is also a compact set in $\ell_2$. Let $\calD\subset \ell_2$ be the convex hull of $\calC\cup\{0\}$. Then both $\calD$ and $\calQ_Q(\calD)$  are compact sets.
By checking the proof of the Lipschitz
extension theorem, we can see that $h^k(I^m\times [-\varepsilon,\varepsilon])\subset \calQ_Q(\calD)$. Now we can apply \cite{BDPG}
Theorem 4.8.2 to conclude that there exists $h\in W^{1}_{p}(I^m\times [-\varepsilon, \varepsilon], \QQl)$ and integers $k_1 < k_2 < \cdots $ such that
\begin{equation}\label{p-conv}
\lim_{j\rightarrow\infty} \int_{I^m\times [-\varepsilon,\varepsilon]} \calG(h^{k_j}, h)^p=0.
\end{equation}
By the lower semicontinuity of $\calE_p$, we have
\begin{equation}
\calE_p(h, I^m \times [-\varepsilon, \varepsilon])\le\liminf_{j\rightarrow\infty} \calE_p(h^{k_j}, I^m\times [-\varepsilon, \varepsilon]).
\label{lsc}
\end{equation}
Since $h^k( \cdot, \varepsilon) ={\bf h}^{k}_1\rightarrow g_1$ in $L^p(I^m)$ and $h^k( \cdot, - \varepsilon)={\bf h}^{k}_2\rightarrow g_2$
in $L^p(I^m)$ as $k\rightarrow\infty$, it follows from \cite{BDPG} Theorem 4.7.3 that
$h( \cdot, \varepsilon)=g_1$ on $B\times \{\varepsilon\}$ and $h( \cdot, -\varepsilon) =g_2$ on $B\times\{-\varepsilon\}$ in the sense of traces.
Finally, by sending $k=k_j$ to $\infty$ in (\ref{p-energy-est}), we see that $h$ satisfies the inequality (\ref{hybrid_ineq2}).
The proof is now complete.
\end{proof}

Now we are ready to prove Theorem \ref{ext}.

\begin{proof}[Proof of Theorem \ref{ext}] The idea is motivated by Luckhaus \cite{Lu}.
We first decompose $\partial B$ into ``cells'' of diameter $\varepsilon$ as follows. Since $B$ is bilipschitz isomorphic to the open unit
cube its boundary can be decomposed into open cubes of side length less than $\varepsilon$ and dimension ranging between $0$ and $m-1$. This gives a partition
$$\partial B=\bigcup_{j=1}^{m-1} \bigcup_{i=1}^{k_j} e_i^j, \ e_i^j\cap e_{i'}^{j'}=\emptyset \ {\rm{if}}\ i\neq i' \ {\rm{or}}\ j\neq j',$$
and for each $e^j_i$ we have a bilipschitz isomorphism
$$\Phi_i^j: e_i^j \to B^j(0, \varepsilon) \ {\rm{with}}\
\|\nabla\Phi_i^j\|_{L^\infty}+\|\nabla(\Phi_i^j)^{-1}\|_{L^\infty}\le c(m).$$
Here $B_\varepsilon^j$ is the $j$-dimensional open ball centered at $0$ and of radius $\varepsilon$.

Denote by $Q_j:=\cup_{i=1}^{k_j}e_i^j$ the union of $j$ cells. We next use the fact that for any nonnegative measurable function $f$
$$\int_{SO(m)}\,d\sigma\int_{\sigma(Q_j)} f \,d\calH^j=\frac{\calH^j(Q_j)}{\calH^{m-1}(\partial B)}\int_{\partial B} fd\calH^{m-1},$$ where $d\sigma$ is the Haar measure on $SO(m)$. By Fubini's theorem, we can further choose a rotation $\sigma\in SO(m)$ such that for all $j$
\begin{equation} \label{fubini}
\calE_p(g_1, \sigma(Q_j))+\calE_p(g_2, \sigma(Q_j))+\int_{\sigma(Q_j)}\varepsilon^{-p}\calG^p(g_1, g_2)d\calH^j
\le c(m) {\bf K}^p \varepsilon^{j+1-m},
\end{equation}
where ${\bf K}>0$ is the constant defined by
$${\bf K}^p:=\calE_p(g_1, \partial B)+\calE_p\big(g_2, \partial B)+\int_{\partial B}\varepsilon^{-p}\calG(g_1, g_2)^p d\calH^{m-1}.$$
To avoid extra notation, assume $\sigma={\rm{id}}$. Now decompose $B\setminus B(0,{1-\varepsilon})$ into cells
$$\hat{e}_i^j=\left\{z\in \mathbb R^n\ : \ \frac{z}{\|z\|}\in e_i^j, \ 1-\varepsilon<\|z\|<1 \right\}, \ 0\le j\le m-1, \ 1 \leq i\le k_j.$$
For $j\le m_p$, we use the fact that $\hat{e}_i^j$ is bilipschitz isomorphic to $B^j(0, \varepsilon) \times [-\varepsilon, \varepsilon]$, \emph{i.e.} there exist
$$\Psi_i^j: {\hat e}_i^j\to B^j(0,\varepsilon)\times [-\varepsilon, \varepsilon] \ {\rm{with}}\
\|\nabla\Psi_i^j\|_{L^\infty}+\|\nabla(\Psi_i^j)^{-1}\|_{L^\infty}\le c(m).
$$
We can apply Lemma \ref{1d-ext} to find an extension map $h_i^j\in W^{1}_{p}\big({\hat e}_i^j,  \QQl\big)$ such that
\begin{equation}\label{trace1}
h_i^j(x,1)=g_1(x), \ h_i^j\big((1-\varepsilon)x\big)=g_2(x), \ \forall x\in e_i^j,
\end{equation}
and
\begin{equation}\label{p-energy-est3}
\calE_p(h_i^j, {\hat e}_i^j)\le C\varepsilon\big(\calE_p(g_1, e_i^j)+\calE_p\big(g_2, e_i^j)+\int_{e_i^j}\varepsilon^{-p}\calG^p(g_1(x), g_2(x))d\calH^j\big).
\end{equation}
Moreover, we can see from the proof of Lemma \ref{1d-ext} that if for $1\le i<i'\le k_j$, the two cells ${\hat e}_i^j$ and ${\hat e}_{i'}^j$  share a common $j$-face, then
one can ensure from the construction of extensions that $h_i^j=h_{i'}^j$ on $\partial{\hat e}_i^j\cap\partial{\hat e}_{i'}^j$ in the sense of traces.
Denote $\hat{Q}^j=\cup_{i=1}^{k_j} \hat{e}_i^j$.
Then we can glue all $h_i^j$ together by letting $ h^j_{|\hat{e}_i^j}=h_i^j$ for $1\le i\le k_j$  to obtain an
extension map $h^j\in W^{1}_{p}\big(\hat{Q}^j, \QQl\big)$ such that
\begin{equation}\label{trace2}
h^j(x,1)=g_1(x), \ h^j\big(\frac{x}{1-\varepsilon}\big)=g_2(x), \ \forall x\in Q^j,
\end{equation}
and
\begin{equation}\label{p-energy-est4}
\calE_p(h^j, {\hat Q}^j)\le C\varepsilon\big(\calE_p(g_1, Q^j)+\calE_p(g_2, Q^j)+\int_{Q^j}\varepsilon^{-p}\calG^p(g_1(x), g_2(x))d\calH^j\big).
\end{equation}

For $j\ge m_p+1$, we use the fact that $\hat{e}_i^j$ is bilipschitz isomorphic to $B^{j+1}(0, \varepsilon)$, \emph{i.e.} there exist
$$F_i^j: {\hat e}_i^j\to B^{j+1}(0, \varepsilon) \ {\rm{with}}\
\|\nabla F_i^j\|_{L^\infty}+\|\nabla(F_i^j)^{-1}\|_{L^\infty}\le c(m).
$$
Since $j+1\ge m_p+2>p$, we can extend $h^j_i$ inductively from the boundary, homogeneous of degree 0:
$$h^j_i\big((F_i^j)^{-1}(z)\big)=h^j_i\big((F_i^j)^{-1}(\frac{\varepsilon z}{\|z\|})\big), \ z\in B^{j+1}(0,\varepsilon).$$
Then we have
\begin{equation}
\calE_p\big(h^j_i, {\hat{e}_i^j}\big)\le C\big(\int_0^{\varepsilon}\big(\frac{r}{\varepsilon}\big)^{(j+1)-1-p}\,dr\big)\calE_p(h_i^j, \partial{\hat e}_i^j)
\le C\varepsilon \calE_p\big(h_i^j, \partial{\hat e}_i^j\big). \label{p-energy-est5}
\end{equation}
By adding (\ref{p-energy-est5}) over all $1\le i\le k_j$ and applying (\ref{fubini}), we then obtain that for any $j\ge m_p+1$,
$h^j$ satisfies (\ref{trace2}) and
\begin{equation}
\calE_p(h^j, {\hat{Q}^j})\le C\varepsilon^{j+2-m} {\bf K}^p. \label{p-energy-est6}
\end{equation}
Since ${\hat Q}^{m-1}=B\setminus B(0,{1-\varepsilon})$, if we define $h=h^{m-1}$ then $h$ satisfies
$$h_{|\partial B}=g_1, \ h_{|\partial B(0,{1-\varepsilon})}=g_2\big(\frac{\cdot}{1-\varepsilon}\big),$$
and
$$
\calE_p(h, B\setminus B(0,{1-\varepsilon}))\le C\varepsilon {\bf K}^p.
$$
Hence the conclusions hold.  The proof is complete.
\end{proof}

\section{Proof of Theorem \ref{holder_reg}}

In this section, we will provide a proof of Theorem \ref{holder_reg} by induction on $Q\in \mathbb Z_+$. The idea is similar to that by Almgren \cite{Almgren},
but we follow closely the presentation by \cite{DS}.

For $v=\oplus_{i=1}^Q\lseg v_i\rseg\in \QQl$, we define the diameter and splitting distance of $v$ as
$$
{d}(v)=\max_{1\le i,j\le Q}\|v_i-v_j\|, \ {s}(v)=\min_{1\le i,j\le Q}\{|v_i-v_j|:  v_i\neq  v_j\}.
$$
If $v=Q\lseg v_0\rseg$ for some $v_0\in \ell_2$, then we define ${s}(v)=+\infty$.

First we need to construct a Lipschitz retraction map from $\QQl$ to $B_\calG(v,r)$, with Lipschitz norm no more than $1$.

\begin{Proposition} \label{lip_retra} For $v\in \QQl$ and
$0<r<s(v)/4<\infty$, there exists a Lipschitz map $\Phi: \QQl\to B_\calG(v,r)$ such that
$\Phi(u)=u$ for any $u\in B_\calG(u,r)$ and ${\rm{Lip}}(\Phi)\le 1$.
\end{Proposition}

\begin{proof} Write $v=\oplus_{j=1}^J k_j \lseg v_j\rseg$ such that $J\ge 2$ and $\|v_i-v_j\|>4r$ for $i\neq j$. If $\calG(u, v)<2r$, then we have
that $u=\oplus_{j=1}^J u_j$ with $u_j=\oplus_{l=1}^{k_j}\lseg u_{l,j}\rseg \in B_\calG (k_j\lseg v_j\rseg, 2r) \subset \calQ_{k_j}(\ell_2)$ for $1\le j\le J$.
Now we can define a Lipschitz retraction map $\Phi:\QQl\to B_\calG(v, r)$ by letting
\begin{equation}\label{retraction}
\Phi(u)=\begin{cases}\displaystyle\bigoplus_{j=1}^J\bigoplus_{l=1}^{k_j}
\lseg\frac{2r-\calG(u,v)}{\calG(u,v)}(u_{l,j}-v_j)+v_j\rseg,
& u\in B_\calG(v, 2r)\setminus B_\calG(v,r),\\
v, & u\in \QQl\setminus B_\calG(v,2r),\\
u, & u\in B_\calG(v,r).
\end{cases}
\end{equation}
It is readily seen that $\Phi$ is an identity map in $B_\calG(v,r)$ and satisfies
$$\calG(\Phi(u_1), \Phi(u_2))\le \calG(u_1, u_2), \ \forall u_1, u_2\in \QQl.$$
Thus $\Phi$ has Lipschitz norm at most $1$.
\end{proof}

\begin{Lemma} \label{large_separation}
 For any $0<\varepsilon<1$, set $\beta(\varepsilon, Q)=\big(\frac{\varepsilon}3\big)^{3^Q}$.
Then, for any $P\in \QQl$ with ${s}(P)<+\infty$, there exists a point $\widetilde{P}\in \QQl$ such that
\begin{equation}\label{large_separation1}
\begin{cases}
\beta(\varepsilon, Q){d}(P)\le {s}(\widetilde P)<+\infty,\\
\ \ \ \  \mathcal G_2(\widetilde P, P)\le\varepsilon {s}(\widetilde P).
\end{cases}
\end{equation}
\end{Lemma}
\begin{proof} The proof can be done exactly in the same way as Lemma 3.8 of \cite{DS} page 35. Here we omit it. \end{proof}

\begin{Proposition} \label{split} Assume $Q \geq 2$. There exists $\alpha(Q)>0$ such that if $u\in W^{1}_{p}(\partial B(0,r), \QQl)$
satisfies that  for some $P \in \QQl$, $\calG(u(x), P)\le\alpha(Q) d(P)$ for $\calH^{m-1}$ a.e. $x\in\partial B(0,r)$,
then there exist $1\le K, L\le Q-1$ with $K+L=Q$ and two  functions $v\in W^{1}_{p}(\partial B(0,r), \calQ_K(\ell_2))$
and $w\in W^{1}_{p}(\partial B(0,r), \calQ_L(\ell_2))$ so that $u= v \oplus w$ a.e. in $\partial B(0,r)$.
\end{Proposition}

\begin{proof} Set $\varepsilon=1/9$ and $\alpha(Q)=\varepsilon\beta(\varepsilon, Q)=\frac19 (27)^{-3^Q}$.
From Lemma \ref{large_separation}, we find a point $\widetilde{P}\in\QQl$ satisfying  (\ref{large_separation1}).
Hence we have that for $\calH^{m-1}$ a.e. $x\in\partial B(0,r)$,
$$\calG(u(x),\widetilde P)\le \calG(u(x),P)+\calG(P,\widetilde P)\le \alpha(Q) d(P)+\frac{s(\widetilde P)}9
\le \frac{2s(\widetilde P)}9<\frac{s(\widetilde P)} 4.$$
Since $s(\widetilde P)<+\infty$, there exists $2\le J\le Q$ such that $\widetilde P=\oplus_{j=1}^J k_j \lseg\widetilde P_j\rseg\in\QQl$
with the $\widetilde P_j$'s all different. Therefore, there exists $J$ functions
$$u_j:\partial B(0,r)\to B_\calG(k_j\lseg\widetilde P_j\rseg, 2r)\subset\calQ_{k_j}(\ell_2)$$
such that $u=\oplus_{j=1}^J u_j$ holds $\calH^{m-1}$ a.e. in $\partial B(0,r)$. Since $u\in W^{1}_{p}(\partial B(0,r), \QQl)$, it
follows that $u_j\in W^{1}_{p}(\partial B(0,r), \calQ_{k_j}(\ell_2))$ for $1\le j\le J$. The proof is complete. \end{proof}

Now we are ready to give a proof of Theorem \ref{holder_reg}.

\begin{proof}[Proof of Theorem \ref{holder_reg}] The key step is to establish the following decay property: there exists $\eta_0>0$ depending
on $p, m, Q$ such that for any $u\in W^{1}_{p}(\partial B(0,r), \QQl)$,
\begin{equation}\label{p-inf2}
\calC(p, Q, u_{|\partial B(0,r)})\le \big(\frac{1}{m-p}-2\eta_0\big) r\calE_p(u, \partial B(0,r)).
\end{equation}
By scalings, one can see that if (\ref{p-inf2}) holds for $r=1$, then it holds for all $r>0$.
We will prove (\ref{p-inf2}) based on an induction on $Q$. For $Q=1$, it is clear that (\ref{p-inf2}) follows from Corollary \ref{holder_reg1}.
Let $Q\ge 2$ be fixed and assume that (\ref{p-inf2}) holds for every $Q^*<Q$. Assume, furthermore, that
$$d(\overline{u})^p>M\calE_p(u, \partial B)$$
for some large constant $M>1$, which will be chosen later.
Apply Lemma \ref{large_separation} with $\varepsilon=\frac1{16}$ and  $P=\overline{u}$, we obtain that there are $2\le J\le Q$
and a point $\widetilde {P}=\oplus_{j=1}^J k_j \lseg Q_j\rseg\in \QQl$ such that
\begin{equation}\label{sep1}
 \beta d(\overline{u})<s(\widetilde P)=\min\{\|Q_i-Q_j\|: \ i\neq j\},
\end{equation}
\begin{equation}\label{sep2}
\calG(\widetilde P, \overline{u})\le \frac{s(\widetilde P)}{16},
\end{equation}
where $\beta=\beta(1/16, Q)$ is the constant given by Lemma \ref{large_separation}.
Let $\Phi: \QQl\to B_\calG(\widetilde P,s(\widetilde P)/8)$ be the Lipschitz contraction map given by Proposition \ref{lip_retra}.
For a small $\eta>0$,  define
$$h : x \in B(0,1-\eta) \mapsto \Phi\big(u\big(\frac{x}{1-\eta}\big)\big) \in \QQl.$$
Then we have that $\Phi(u)\in W^{1}_{p}\big(\partial B(0,{1-\eta}), B_\calG(\widetilde P, s(\widetilde P)/8)\big)$. Apply  Proposition \ref{split}, we conclude that
there exist $1\le K, L\le Q-1$, with $K+L=Q$, and $h_1\in W^{1}_{p}(\partial B(0,{1-\eta}), \calQ_K(\ell_2))$,
$h_2\in W^{1}_{p}(\partial B(0,{1-\eta}), \calQ_L(\ell_2))$ such that $h= h_1 \oplus h_2$ in $\partial B(0,{1-\eta})$.
By the induction hypothesis, we have
\begin{equation}\label{p-min4}
\begin{cases}\calC(p, K, h_{1|\partial B(0,{1-\eta})})\le \left(\frac{1}{m-p}-6\eta_0\right) (1-\eta) \calE_p(h_1, \partial B(0,{1-\eta})),
\\
\calC(p, L, h_{2|\partial B(0,{1-\eta})})\le \left(\frac{1}{m-p}-6\eta_0\right) (1-\eta) \calE_p(h_2, \partial B(0,{1-\eta})).
\end{cases}
\end{equation}
By (\ref{p-min4}), there exist $\hat{h}_1\in W^{1}_{p}(B(0,{1-\eta}), \calQ_K(\ell_2))$ and  $\hat{h}_2\in W^{1}_{p}(B(0,{1-\eta}), \calQ_L(\ell_2))$
such that
\begin{equation}\label{p-min5}
\begin{cases}
\calE_p(\hat h_1, B(0,{1-\eta}))\le  \left(\frac{1}{m-p}-6\eta_0\right) (1-\eta) \calE_p(h_1, \partial B(0,{1-\eta}))+\eta_0 \calE_p(u,\partial B),\\
\calE_p(\hat h_2, B(0,{1-\eta}))\le  \left(\frac{1}{m-p}-6\eta_0\right) (1-\eta) \calE_p(h_2, \partial B(0,{1-\eta}))+\eta_0 \calE_p(u,\partial B).
\end{cases}
\end{equation}
Define $\hat h=\hat h_1 \oplus \hat h_2$. Then $\hat h\in W^{1}_{p}(B(0,{1-\eta}), \QQl)$ satisfies
$\hat h=h$ on $\partial B(0,{1-\eta})$ and
\setlength\arraycolsep{1.4pt}
\begin{eqnarray}\label{p-min6}
\calE_p(\hat h, B_{1-\eta}) & \le &  \big(\frac{1}{m-p}-6\eta_0\big) (1-\eta)
 \left(\calE_p(h_1, \partial B(0,{1-\eta}))+\calE_p(h_2, \partial B(0,{1-\eta}))\right)
  \nonumber \\
&& +2\eta_0 \calE_p(u,\partial B)
\nonumber\\
&=& \big(\frac{1}{m-p}-6\eta_0\big) (1-\eta) \calE_p(h, \partial B(0,{1-\eta}))+2\eta_0 \calE_p(u,\partial B)\nonumber\\
&= & \big(\frac{1}{m-p}-6\eta_0\big) \calE_p(\Phi(u), \partial B)+2\eta_0 \calE_p(u,\partial B)\nonumber\\
&\le &  \big(\frac{1}{m-p}-6\eta_0\big) \calE_p(u, \partial B)+2\eta_0 \calE_p(u,\partial B) \nonumber \\
& = & \big(\frac{1}{m-p}-4\eta_0\big) \calE_p(u, \partial B)
\end{eqnarray}
where we have used the fact that ${\rm{Lip}}(\Phi)\le 1$ in the last inequality.

Now let $\hat g\in W^{1}_{p}(B\setminus B(0,{1-\eta}), \QQl)$ be an extension of $$\Phi\big(u\big(\frac{\cdot}{1-\eta}\big)\big)\in W^{1}_{p}(\partial B(0,{1-\eta}), \QQl)$$
and $u\in W^{1}_{p}(\partial B, \QQl)$ as in Lemma \ref{1d-ext}. Define $\hat u\in W^{1}_{p}(B, \QQl)$ by
$$
\hat u=\begin{cases} \hat h &\  {\rm{in}}\ B(0,{1-\eta})\\
\hat  g  &\  {\rm{in}}\ B\setminus B(0,{1-\eta}).
\end{cases}
$$
Then we have
\setlength\arraycolsep{1.4pt}
\begin{eqnarray}
&&\calE_p(\hat u, B)\nonumber\\
& = & \calE_p(\hat h, B(0,{1-\eta}))+\calE_p(\hat g, B\setminus B(0,{1-\eta}))\nonumber\\
&\le & \big(\frac{1}{m-p}-4\eta_0+C\eta\big) \calE_p(u, \partial B) +\frac{C}{\eta}\int_{\partial B}\calG^p(u, \Phi(u))d\calH^{m-1}. \label{p-min7}
\end{eqnarray}
Now we need to estimate $\int_{\partial B}\calG^p(u, \Phi(u))d\calH^{m-1}$.  Define
$$E:=\left\{x\in \partial B: u(x)\neq \Phi(u(x))\right\}=\left\{x\in \partial B: u(x) \not\in B_\calG\big(\widetilde{P},{\frac{s(\widetilde P)}8}\big)\right\}.$$
Since $\Phi(\overline{u})=\overline{u}$, we have
$$\calG(\overline{u}, \Phi(u(x)))\le \calG(\overline{u}, u(x)),  \ \forall x\in\partial B.$$
Hence we have
\setlength\arraycolsep{1.4pt}
\begin{eqnarray}
\int_{\partial B}\calG^p(u, \Phi(u))d\calH^{m-1} & = & \int_{E}\calG^p(u, \Phi(u))d\calH^{m-1}\nonumber\\
&\le & C\int_{E}\left(\calG^p(u, \overline{u})+\calG^p(\overline{u},\Phi(u))\right)d\calH^{m-1}\nonumber\\
&\le & C\int_{E}\calG^p(u, \overline{u})\,d\calH^{m-1}\nonumber\\
&\le & C\|\calG(u, \overline{u})\|_{L^{p^*}(\partial B)}^{\frac{p}{p^*}}(\calH^{m-1}(E))^{1-\frac{p}{p^*}}\nonumber\\
&\le & C \calE_p(u, \partial B)(\calH^{m-1}(E))^{1-\frac{p}{p^*}},
\end{eqnarray}
where $p^*$ is the Sobolev exponent of $p$ in $\mathbb R^{n-1}$:
$$p^*=\begin{cases} \frac{(m-1)p}{m-1-p} & {\rm{if}}\ p<m-1,\\
 {\rm{any}} \ q\in (p, +\infty) & {\rm{if}}\ p\ge m-1.\end{cases}
$$
For any $x\in E$, we have
$$\calG(u(x), \overline{u})\ge \calG(u(x), \widetilde P)-\calG(\widetilde P, \overline{u})
\ge \frac{s(\widetilde P)}8-\frac{s(\widetilde P)}{16}=\frac{s(\widetilde P)}{16}.$$
So we have that
\setlength\arraycolsep{1.4pt}
\begin{eqnarray}\label{size_of_E}
\calH^{m-1}(E) & \le &  \calH^{m-1}\Big(\big\{x\in\partial B : \calG(u(x), \overline{u})\ge \frac{s(\widetilde P)}{16}\big\}\Big)\nonumber\\
&\le & \frac{C}{s^p(\widetilde P)}\int_{\partial B}\calG^p(u(x), \overline{u}) \nonumber \\
& \le &  \frac{C}{d^p(\overline{u})}\calE_p(u, \partial B) \nonumber \\
& \le & \frac{C}{M}.
\end{eqnarray}
Therefore we obtain
\begin{equation}\label{size_of_E1}
\int_{\partial B}\calG^p(u, \Phi(u))d\calH^{m-1}\le C\left(\frac{C}{M}\right)^{1-\frac{p}{p^*}}
\calE_p(u, \partial B).
\end{equation}
Substituting (\ref{size_of_E1}) into (\ref{p-min7}), we find that
\begin{equation}
\calE_p(\hat u, B)
\le\left(\frac{1}{m-p}-4\eta_0+C\eta+\frac{C}{\eta} M^{\frac{p}{p^*}-1}\right)
\calE_p(u, \partial B). \label{p-min8}
\end{equation}
Now we first choose $\eta=\eta_0/C$ and then choose
$$
M=\big(\frac{C^2}{\eta_0^2}\big)^{\frac{p^*}{p^*-p}}
$$
so that (\ref{p-min8}) yields
\begin{equation}
\calC(p, Q, u_{|\partial B})\le \calE_p(\hat u, B)
\le \big(\frac{1}{m-p}-2\eta_0\big)
\calE_p(u, \partial B), \ {\rm{if}}\ d^p(\overline{u})>M \calE_p(u, \partial B). \label{p-min9}
\end{equation}
On the other hand, Lemma \ref{p-min0} implies that
\begin{equation}
\calC(p, Q, u_{|\partial B})\le \calE_p(\hat u, B)
\le \big(\frac{1}{m-p}-2\eta_0\big)
\calE_p(u, \partial B), \ {\rm{if}}\ d^p(\overline{u})\le M \calE_p(u, \partial B). \label{p-min10}
\end{equation}
Combining (\ref{p-min9}) and (\ref{p-min10}) yields that (\ref{p-inf2}) holds for $r=1$.  Note that (\ref{p-inf2})
for all $r\neq 1$ follows from (\ref{p-inf2}) for $r=1$ by simple scalings.

Since $u$ is $(\omega,p)$-Dir-minimizing in $U$,  by (\ref{p-inf2}) we have that for any ball $B(x,r)\subset U$,
\setlength\arraycolsep{1.4pt}
\begin{eqnarray}\label{p-min11}
\calE_p(u, B(x,r))&\le& (1+\omega(r))\calC(p,Q, u_{|\partial B(x,r)})\nonumber\\
&\le& (1+\omega(r))\big(\frac{1}{m-p}-2\eta_0\big) r \calE_p(u, \partial B(x,r)).
\end{eqnarray}
Since $\lim_{r\downarrow 0}\omega(r)=0$, there exists $r_0>0$ such that
$$(1+\omega(r))\big(\frac{1}{m-p}-2\eta_0\big)\le \frac{1}{m-p}-\eta_0, \ \forall 0<r\le r_0.$$
Thus we have that
\begin{equation}\label{p-min12}
\calE_p(u, B(x,r))\le \big(\frac{1}{m-p}-\eta_0\big) r \calE_p(u, \partial B(x,r))
\end{equation}
holds for all $B(x,r)\subset U$ with $0<r\le r_0$. It is standard that integrating (\ref{p-min12}) over $r$ yields
that
\begin{equation}\label{p-decay}
\frac{1}{r^{m-p+\eta_0}}\calE_p(u, B(x,r))\le \frac{1}{r_0^{m-p+\eta_0}}\calE_p(u, B(x,r_0)), \ \forall B(x,r_0)\subset U, \ 0<r\le r_0.
\end{equation}
This, combined with the Morrey decay lemma \cite{Morrey} for $\QQl$-valued functions, implies that
$u\in C^{\eta_0/p}(U)$ and
$$\|u\|_{C^{\eta_0/p}(B(x_0,r)) } \le C\frac{1}{r_0^{m-p+\eta_0}}\calE_p(u, B(x_0,r_0))$$
holds for $B(x_0,r_0)\subset U$ and $0<r\le r_0$. This completes the proof of Theorem \ref{holder_reg}.
\end{proof}

\section{Squeeze stationary maps}

We say that a map $f \in W^1_2(U, \QQl)$ is squeeze stationary in $U$ whenever for every $X \in C^\infty_c(U, \R^m)$, we have
\begin{equation}\label{def_squeeze}
\frac{d}{dt}\Big|_{t=0} \int_U \lno D(f \circ \Phi_t) \rno^2 = 0,
\end{equation}
where $\Phi_t(x) = x + tX(x)$ is a diffeomorphism of $U$ to itself for small values of $t$.

\begin{Proposition}
A map $f \in W^1_2(U, \QQl)$ is squeeze stationary if and only if for every vector field $X \in C^\infty_c(U, \R^m)$ one has:
\begin{equation}\label{euler_lagrange_squeeze}
2 \sum_{i=1}^Q \int_U \langle Df(y), Df(y) \circ DX(y) \rangle dy - \int_U \lno Df(y)\rno^2 \mathrm{div}\, X(y) dy = 0.
\end{equation}
\end{Proposition}

\begin{proof}
We prove (\ref{euler_lagrange_squeeze}) by computing (\ref{def_squeeze}). Note that
\setlength\arraycolsep{1.4pt}
\begin{eqnarray*}
\lno D(f \circ \Phi_t)(x) \rno^2 & = & \sum_{i=1}^Q \| D(f_i \circ \Phi_t)(x) \|_{{\rm{HS}}}^2 \\
& = & \sum_{i=1}^Q \|Df_i(\Phi_t(x)) \circ (D\Phi_t(x)) \|_{{\rm{HS}}}^2 \\
& = & \sum_{i=1}^Q \|Df_i(\Phi_t(x)) \circ (\mathrm{id}_{\R^m} + t DX(x))\|^2_{{\rm{HS}}} \\
& = & \sum_{i=1}^Q \|Df_i(\Phi_t(x)) + t Df_i(\Phi_t(x)) \circ (DX(x))\|^2_{\rm{HS}}.
\end{eqnarray*}
We now change variable $y = \Phi_t(x)$ in (\ref{def_squeeze}):
\setlength\arraycolsep{1.4pt}
\begin{eqnarray*}
&&\int_U \lno D(f\circ \Phi_t)(x) \rno^2 dx\\
& = &\sum_{i = 1}^Q \| Df_i(y) + tDf(y) \circ DX(\Phi_t^{-1}(y))\|^2 |\mathrm{det}(D\Phi_t^{-1}(y)|dy \\
& = & \sum_{i=1}^Q \int_U \|A_i(y) + tB_i(y,t)\|_{{\rm{HS}}}^2 C_i(y,t) dy,
\end{eqnarray*}
where
$$ A_i(y) = Df_i(y), B_i(y,t) = Df(y) \circ DX(\Phi_t^{-1}(y)), C_i(y,t) = |\mathrm{det}(D\Phi_t^{-1}(y)|.$$
It remains to differentiate with respect to $t$:
\setlength\arraycolsep{1.4pt}
\begin{eqnarray*}
&&\|A_i(y) + tB_i(y,t) \|_{{\rm{HS}}}^2 =  \langle A_i(y) + tB_i(y,t), A_i(y) + tB_i(y,t) \rangle \\
& = & \|A_i(y)\|^2_{{\rm{HS}}} + 2t \langle A_i(y), B_i(y,t) \rangle + t^2 \langle B_i(y)\rangle_{{\rm{HS}}}^2 \\
& = & \|A_i(y)\|^2_{{\rm{HS}}} + 2t \langle A_i(y), B_i(y,0) \rangle + o(t),
\end{eqnarray*}
thus
\setlength\arraycolsep{1.4pt}
\begin{eqnarray*}
\frac{d}{dt}\Big|_{t=0}\left\|A_i(y) + tB_i(y)\right\|^2_{{\rm{HS}}}
& = & 2\langle A_i(y), B_i(y,0)\rangle \\
& = & 2 \langle Df_i(y), Df_i(y) \circ DX(y) \rangle.
\end{eqnarray*}
Furthermore, $y = \Phi_t(\Phi_t^{-1}(y))$ so that
$$ \mathrm{id}_{\R^m} = D\Phi_t(\Phi^{-1}_t(y)) \circ D\Phi_t^{-1}(y), $$
and hence
$$ \mathrm{det}\, D\Phi_t^{-1}(y) = \frac{1}{h(y,t)} \qquad \textrm{where}\qquad h(y,t) := \mathrm{det}\, D\Phi_t(\Phi_t^{-1}(y)).$$
It follows that
$$
\frac{d}{dt}\Big|_{t=0}\left(\mathrm{det}\, D\Phi_t^{-1}(y) \right)= - \frac{1}{h(y,0)} \frac{\partial h}{\partial t}(y,0). $$
Clearly $h(0) = 1$. Next,
\setlength\arraycolsep{1.4pt}
\begin{eqnarray*}
\mathrm{det}\, D\Phi_t(\Phi_t^{-1}(y)) & = & \mathrm{det}(\mathrm{id}_{\R^m} + t DX(\Phi_t^{-1}(y)) \\
& = & 1 + t \mathrm{tr}\, DX(\Phi_t^{-1}(y)) + o(t),
\end{eqnarray*}
whence
$$ \frac{\partial h}{\partial t}(y,0) = \mathrm{tr}\, DX(y) = \mathrm{div}\,X(y).$$
Putting together everything, we obtain Equation (\ref{euler_lagrange_squeeze}).
\end{proof}

\begin{Proposition}\label{energy_nondecreasing}
Let $f \in W^1_2(U, \QQl)$ be squeeze stationary and $a \in U$. It follows that the function $ \Theta_a : (0, \mathrm{dist}(a, \partial U)) \to \R_+$ defined by
$$ \Theta_a(r) := \frac{1}{r^{m-2}} \int_{B(a,r)} \lno Df \rno^2 $$
is absolutely continuous and nondecreasing. In fact,
\begin{equation}\label{derivative_theta}
 \Theta_a'(r) = \frac{2}{r^{m-2}} \int_{\partial B(a,r)} \sum_{i=1}^Q \big\| \frac{\partial f_i}{\partial \nu} \big\|^2d\calH^{m-1}.
\end{equation}
In other words, the Radon measure
\begin{equation}\label{monotonic_measure} A \mapsto \int_A \lno Df\rno^2 \end{equation}
is $(m-2)$ monotonic.
\end{Proposition}

\begin{proof}
It is clear that $\Theta_a$ is absolutely continuous because the measure in (\ref{monotonic_measure}) is absolutely continuous with respect to the Lebesgue measure $\calL^m$ and $\calL^m(B(a,r+h) \setminus B(a,r)) \leq Ch$.
For simplicity, assume $a = 0$ and write $\Theta(r)$ for $\Theta_a(r)$.
We now plug in equation (\ref{euler_lagrange_squeeze}) a vector field
$$ X(x) = \chi(\|x \|)x,$$ where $\chi \in C^\infty_c(\R)$ is constant in a neighborhood of 0. For every $i,j = 1, \dots, m$ one has
\setlength\arraycolsep{1.4pt}
\begin{eqnarray*}
\frac{\partial}{\partial x_j} \langle X(x), e_i \rangle & = & \frac{\partial}{\partial x_j} \left(\chi(\|x\|) x_i \right) \\
& = & \chi(\|x\|) \delta_{ij} + \chi'(\|x\|) \frac{x_ix_j}{\|x\|}.
\end{eqnarray*}
In particular,
$$ \mathrm{div}\, X(x) = m \chi(\|x\|) + \chi'(\|x\|) \|x\|. $$
We now compute the first term in (\ref{euler_lagrange_squeeze}):
\setlength\arraycolsep{1.4pt}
\begin{eqnarray*}
\sum_{i=1}^Q \langle Df_i(x), Df_i(x) \circ DX(x) \rangle & = & \chi(\|x\|)
\sum_{i=1}^Q \| Df_i(x)\|_{{\rm{HS}}}^2 \\
& & + \chi'(\|x\|) \sum_{i=1}^Q \big\langle Df_i(x), Df_i(x) \circ \big(x \otimes \frac{x}{\|x\|} \big) \big\rangle \\
& = & \chi(\|x\|) \lno Df(x)\rno^2 \\
& & + \chi'(\|x\|) \|x\| \sum_{i=1}^Q \langle Df_i(x), Df_i(x) \circ A\rangle,
\end{eqnarray*}
where $A$ is the matrix $u_1 \otimes u_1$ for $u_1 := x / \|x\| \in \mathbb{S}^{m-1}$, \emph{i.e.} $A_{ij} = \langle u_1,e_i \rangle \langle u_1,e_j \rangle $. One easily sees that $Au_1 = u_1$ and that $Av = 0$ whenever $\langle u_1,v \rangle = 0$. Now we complete $u_1$ to an orthonormal basis $(u_1, \dots, u_m)$. It follows that
\setlength\arraycolsep{1.4pt}
\begin{eqnarray*}
\langle Df_i(x), Df_i(x) \circ A \rangle & = & \sum_{j=1}^m \langle Df_i(x)(u_j), Df_i(x)(A(u_j)) \rangle \\
& = & \langle Df_i(x)(u_1), Df_i(x)(u_1) \rangle \\
& = & \left\| \frac{\partial f_i}{\partial \nu} (x) \right\|^2.
\end{eqnarray*}
We infer from (\ref{euler_lagrange_squeeze}) that
$$
 2\int_U \big(  \lno Df(x)\rno^2 \chi(\|x\|) + \chi'(\|x\|) \|x\| \sum_{i=1}^Q \big\| \frac{\partial f_i}{\partial \nu} (x) \big\|^2\big) dx $$
 $$ - \int_U \big(\lno Df(x) \rno^2 \big(m\chi(\|x\| + \|x\| \chi'(\|x\|) \big) \big) dx = 0,
$$
that is,
$$ (2-m) \int_U \chi(\|x\|) \lno Df(x) \rno^2 dx + 2 \int_U \chi'(\|x\|)\|x\| \sum_{i=1}^Q \big\| \frac{\partial f_i}{\partial \nu} (x) \big\|^2 dx
$$
$$
= \int_U \chi'(\|x\|) \|x\| \lno Df(x)\rno^2 dx.
$$
Now we fix $r \in (0, \mathrm{dist}(0, \partial U)$ and we let $\{\chi_j\}_{j=1}^\infty$ approach $\ind_{[0,r]}$ so that
$$ (2-m) \int_{B(0,r)} \lno Df \rno^2 - 2r \int_{\partial B(0,r)} \sum_{i=1}^Q
\big\| \frac{\partial f_i}{\partial \nu} \big\|^2 d\calH^{m-1}
=-r\int_{\partial B(0,r)} \lno Df \rno^2,$$
and one finally computes that for a.e $r \in (0, \mathrm{dist}(0, \partial U))$,
\setlength\arraycolsep{1.4pt}
\begin{eqnarray*}
\Theta'(r) & = &
(2-m)r^{2-m-1} \int_{B(0,r)} \lno Df \rno^2 + r^{2-m} \int_{\partial B(0,r)} \lno Df \rno^2 \\
& = & 2r^{2-m} \int_{\partial B(0,r)} \sum_{i=1}^Q
\big\| \frac{\partial f_i}{\partial \nu} \big\|^2d\calH^{m-1} \geq 0.
\end{eqnarray*}
\end{proof}

\section{Squash variations}

Here we consider vertical variations above $x$ whose amplitude depend on $x \in U$. Let
$$
Y : U \times \ell_2 \to \ell_2
$$
be a $C^1$ map such that $U \cap \{x : Y(x, \cdot) \neq 0\}$ is relatively compact in $U$. For $f : U \to \QQl$ we define
$$
(Y \Box f)_t(x) := \bigoplus_{i=1}^Q \lseg f_i(x) + t Y(x, f_i(x)) \rseg.
$$
We say that $f$ is squash stationary if for every such $Y$ one has
$$
\frac{d}{dt}\Big|_{t=0} \int_U \lno D(Y \Box f)_t \rno^2 = 0.
$$
\begin{Proposition}
A map $f \in W_2^1(U, \QQl)$ is squash stationary if and only if for all $Y$,
\begin{equation}\label{euler-lagragne-squash}
\sum_{i=1}^Q \int_U \langle Df_i(x), D_xY(x,f_i(x)) \rangle
+ \sum_{i=1}^Q \int_U \langle Df_i(x), D_yY(x,f_i(x)) \circ Df_i(x) \rangle dx = 0.
\end{equation}
\end{Proposition}

\begin{proof}
The derivation of this Euler-Lagrange equation is much simpler than for squeeze variations. One computes
\setlength\arraycolsep{1.4pt}
\begin{eqnarray*}
&&\lno D(Y \Box f)_t \rno^2 = 
\sum_{i=1}^Q \|D(f_i + tY( \cdot, f_i(\cdot)) \|_{{\rm{HS}}}^2 \\
&& =  \lno Df \rno^2 + 2t\sum_{i=1}^Q \langle Df_i, D_xY(\cdot, f_i(\cdot)) + D_yY(\cdot, f_i(\cdot)) \circ Df_i \rangle + o(t).
\end{eqnarray*}
Integrating over $U$ and then differentiating at $t = 0$ gives equation (\ref{euler-lagragne-squash}).
\end{proof}

\begin{Corollary}
If $f \in W_2^1(U, \QQl)$ is squash stationary and $B(a,r) \subset U$ then
\begin{equation}\label{cor_squash}
\int_{B(a,r)} \lno Df \rno^2 = \int_{\partial B(a,r)} \sum_{i=1}^Q \big\langle
\frac{\partial f_i}{\partial \nu}, f_i\big\rangle d\calH^{m-1}.
\end{equation}
\end{Corollary}

\begin{proof} Assume for simplicity that $a = 0$.
Let $\chi \in C^\infty_c(\R)$ be a function which is constant in a neighborhood of $0$. We will plug
$$ Y(x,y) := \chi(\|x\|) y $$
into equation (\ref{euler-lagragne-squash}). First we compute
$$ D_xY(x,y) = \chi'(\|x\|) \frac{x}{\|x\|} \otimes y, \qquad D_yY(x,y) = \chi(\|x\|) \mathrm{id}_{\R^m}.$$
Thus we have
$$
\sum_{i=1}^Q \int_U \chi'(\|x\|) \big\langle \frac{\partial f_i}{\partial \nu}(x), f_i(x) \big\rangle + \int_U \chi(\|x\|) \lno Df(x) \rno^2 dx = 0.
$$
Now, we fix $r \in (0, \mathrm{dist}(0, \partial U))$ and we let $\{\chi_j\}_{j=1}^\infty$ be an approximation of $\ind_{[0,r]}$, so that
$$
- \sum_{i=1}^Q \int_{\partial B(0,r)} \big\langle \frac{\partial f_i}{\partial \nu}, f_i\big\rangle
d\calH^{m-1}+ \int_{B(0,r)} \lno Df \rno^2 = 0.
$$
\end{proof}

\section{Frequency function}

Let $f \in W_2^1(U, \QQl)$. For $a \in U \subset \R^m$ and $r \in (0, \mathrm{dist}(a, \partial U)$, we define the following quantities
$$
D_a(r) := \int_{B(a,r)} \lno Df \rno^2, \qquad H_a(r) := \int_{\partial B(a,r)} |f|^2, \qquad N_a(r) := \frac{rD_a(r)}{H_a(r)},
$$
where the last quantity, $N_a(r)$, is defined provided that $H_a(r) \neq 0$. $N_a(r)$ is called the \emph{frequency function} at the point $a$.  When $a=0$, we simply denote $D(r)$, $H(r)$, $N(r)$
for $D_0(r)$, $H_0(r)$, $N_0(r)$ respectively.

The following lemma ensures that if $f$ is squash stationary and non zero in a neighborhood of $a$, then $N_a(r)$ is defined for small values of $r$.

\begin{Lemma}
If $f \in W^{1}_2(U, \QQl)$ is squash stationary and if $H_a(r_0) = 0$ for some $a\in U$
and $r_0 \in (0, \mathrm{dist}(a, \partial U))$, then $f \equiv Q \lseg 0 \rseg $ on $B(a, r_0)$.
\end{Lemma}

\begin{proof}
This is a consequence of (\ref{cor_squash}): if $f \equiv 0$ a.e on $\partial B(a,r)$ then $\calE_2(f, B(a,r)) = 0$, thus $f$ vanishes on $B(a,r)$.
\end{proof}

\begin{Lemma}
For $a\in U$, $H_a(r)$ is absolutely continuous in $r$, for any $f \in W^1_2(U, \QQl)$.
\end{Lemma}

\begin{proof} For simplicity, assume $a=0\in U$.  Then
\setlength\arraycolsep{1.4pt}
\begin{eqnarray*}
H(r) & = & \frac{d}{dr} \int_{ B(0,r)} |f(x)|^2 dx \nonumber \\
& = & \frac{d}{dr} \left(r^{m} \int_{ B} |f(rx)|^2 dx \right) \nonumber \\
& = & m r^{m-1} \int_{B} |f(rx)|^2 dx + 2 r^{m} \int_{B} \sum_{i=1}^Q \big\langle f_i(rx), \frac{\partial f_i}{\partial \nu}(rx) \big\rangle dx \nonumber \\
& = & \frac{m }{r} \int_{ B(0,r)} |f|^2 + 2 \sum_{i=1}^Q \int_{ B(0,r)} \big\langle f_i, \frac{\partial f_i}{\partial \nu} \big\rangle,
\end{eqnarray*}
and the last expression is easily seen to be absolutely continuous.
\end{proof}

\begin{Lemma}
If $f$ is squash stationary then for any $a\in U$ and a.e. $r \in (0, \mathrm{dist}(a, \partial U))$,
\begin{equation}\label{derivative_H}
H_a'(r) = \frac{m-1}{r} H_a(r) + 2D_a(r).
\end{equation}
\end{Lemma}

\begin{proof} For simplicity, assume $a=0\in U$. Then
we have
\setlength\arraycolsep{1.4pt}
\begin{eqnarray*}
H'(r) & = & \frac{d}{dr} \int_{\partial B} r^{m-1} |f(rx)|^2 d\calH^{m-1}x \\
& = & \frac{m-1}{r} \int_{\partial B(0,1)} r^{m-1} |f(rx)|^2 d\calH^{m-1}x \\
&&+ 2 r^{m-1}\int_{\partial B(0,r)} \sum_{i=1}^Q \big\langle f_i(rx), \frac{\partial f_i}{\partial \nu} (rx) \big\rangle d\calH^{m-1}x \\
& = & \frac{m-1}{r} \int_{\partial B(0,r)} |f|^2d\calH^{m-1} + 2 \int_{\partial B(0,r)} \sum_{i=1}^Q
\big\langle f_i, \frac{\partial f_i}{\partial \nu} \big\rangle d\calH^{m-1}.
\end{eqnarray*}
Now using (\ref{cor_squash}) we prove (\ref{derivative_H}).
\end{proof}

\begin{Theorem} \label{monotonic_frequency}
If $f$ is squeeze and squash stationary and  for $a\in U$, $H_a(r_0) \neq 0$ for some $r_0\in (0, \mathrm{dist}(a, \partial U))$, then $N_a'(r) \geq 0$ for a.e. $r \leq r_0$.
\end{Theorem}

\begin{proof} For simplicity, assume $a=0\in U$.
Select $r \in (0,r_0]$ at which $D$ and $H$ are differentiable. Then, according to (\ref{derivative_H}), (\ref{derivative_theta}) and (\ref{cor_squash}),
\setlength\arraycolsep{1.4pt}
\begin{eqnarray*}
N'(r) & = &
\frac{D(r)}{H(r)} + r \frac{D'(r)}{H(r)} - rD(r) \frac{H'(r)}{H^2(r)} \\
& = & \frac{D(r)}{H(r)} + r\frac{D'(r)}{H(r)} - (m-1)\frac{D(r)}{H(r)} - 2r \frac{D^2(r)}{H^2(r)} \\
& = & \frac{(2-m) D(r) + rD'(r)}{H(r)} - 2r \frac{D^2(r)}{H^2(r)} \\
& = & \frac{r^{m-1} \Theta'(r) }{H(r)} - 2r \frac{D^2(r)}{H^2(r)} \\
& = & \frac{2r}{H(r)} \int_{\partial B(0,r)} \sum_{i=1}^Q \big\| \frac{\partial f_i}{\partial \nu} \big\|^2 - \frac{2r}{H(r)^2} \big(\int_{\partial B(0,r)} \sum_{i=1}^Q \big\langle
f_i , \frac{\partial f_i}{\partial \nu} \big\rangle \big)^2.
\end{eqnarray*}
Thus,
$$
N'(r) = \frac{2r}{H^2(r)} \big[\int_{\partial B(0,r)} \sum_{i=1}^Q \|f_i\|^2 \int_{\partial B(0,r)} \sum_{i=1}^Q \big\| \frac{\partial f_i}{\partial \nu} \big\|^2 -
\big(\int_{\partial B(0,r)} \sum_{i=1}^Q \big\langle
f_i , \frac{\partial f_i}{\partial \nu} \big\rangle \big)^2 \big]
$$
is nonnegative by Cauchy-Schwartz inequality.
\end{proof}

\begin{Corollary}
Assume that $f \in W^1_2(U, \QQl)$ is squeeze and squash stationary and $r_0 \in (0, \mathrm{dist}(a, \partial U))$ for
$a\in U$.
Then, for a.e. $r \leq r_0$, we have
\begin{equation}\label{cor_freq1}
\frac{d}{dr} \ln \frac{H_a(r)}{r^{m-1}} = \frac{2N_a(r)}{r},
\end{equation}
\begin{equation}\label{cor_freq2}
  \left(\frac{r}{r_0}\right)^{2 N_a(r_0)} \frac{H_a(r_0)}{r_0^{m-1}} \leq \frac{H_a(r)}{r^{m-1}} \leq \left( \frac{r}{r_0}\right)^{2N_a(r)} \frac{H_a(r_0)}{r_0^{m-1}},
\end{equation}
\begin{equation}\label{cor_freq3}
 \frac{D_a(r)}{r^{m-2}} \leq \left(\frac{r}{r_0}\right)^{2N_a(r)} \frac{D_a(r_0)}{r^{m-2}} \frac{N_a(r)}{N_a(r_0)}.
\end{equation}
\end{Corollary}

\begin{proof} For simplicity, assume $a=0\in U$.
First we prove (\ref{cor_freq1}). By equation (\ref{derivative_H}), we have
$$ \frac{d}{dr} \frac{H(r)}{r^{m-1}} = \frac{H'(r)}{r^{m-1}} - (m-1) \frac{H(r)}{r^m} = \frac{2D(r)}{r^{m-1}}.$$
Consequently,
$$\frac{d}{dr} \ln \left( \frac{H(r)}{r^{m-1}} \right) = \frac{2D(r)}{r^{m-1}} \frac{r^{m-1}}{H(r)} = \frac{2N(r)}{r}.
$$
We integrate (\ref{cor_freq1}) from $r$ to $r_0$:
\setlength\arraycolsep{1.4pt}
\begin{eqnarray}\label{computation_cor_freq2}
\ln \frac{H(r_0)}{r_0^{m-1}} - \ln \frac{H(r)}{r^{m-1}} & = &
\int_r^{r_0} \frac{d}{d\rho} \ln \frac{H(\rho)}{\rho^{m-1}} d\rho \nonumber \\
& = & \int_r^{r_0} \frac{2N(\rho)}{\rho}d\rho \nonumber \\
& \geq & 2N(r) \int_r^{r_0} \frac{d\rho}{\rho}.
\end{eqnarray}
Therefore
$$
\ln\left(\frac{H(r_0)}{r_0^{m-1}} \frac{r^{m-1}}{H(r)} \right) \geq \ln \left( \frac{r_0}{r} \right)^{2N(r)},
$$
which yields the first inequality in (\ref{cor_freq2}). Similarly, by bounding above $N(\rho)$ by $N(r_0)$ in (\ref{computation_cor_freq2}), one proves the reverse inequality.

By definition of $N(r)$,
$$ \frac{D(r)}{r^{m-2}} = \frac{H(r)}{r^{m-1}} N(r). $$
Therefore, using (\ref{cor_freq2}), one has
$$
\frac{D(r)}{r^{m-2}} \leq \big(\frac{r}{r_0}\big)^{2N(r)} \frac{H(r_0)}{r_0^{m-1}} N(r) =  \big(\frac{r}{r_0}\big)^{2N(r)} \frac{D(r_0)}{r_0^{m-2}} \frac{N(r)}{N(r_0)},
$$ which is (\ref{cor_freq3}).
\end{proof}

\section{Regularity of stationary maps}

\begin{Proposition}\label{stationary_bounded}
If $f \in W_2^1(U, \QQl)$ is squeeze and squash stationary and let $a\in U$ and
$B(a,3r_0) \subset U$, then $f$ is locally essentially bounded on $B(a,r_0)$. Specifically,
\begin{equation}\label{control_boundedness_stationary}
\|f\|_{L^\infty(B(a,r_0))}^2 \leq \frac{C}{r_0^m} \int_{B(a,3r_0)} |f|^2.
\end{equation}
\end{Proposition}

\begin{proof} Assume $a=0\in U$.  Let $x_0\in B(0,r_0)$. Then we have
\setlength\arraycolsep{1.4pt}
\begin{eqnarray*}
\int_{r_0}^{2r_0} d\rho \int_{\partial B(x_0,\rho)} |f|^2 d\calH^{m-1}
& = & \int_{B(x_0,2r_0) \setminus B(x_0,r_0)} |f|^2 \\
& \leq & \int_{B(0,3r_0)} |f|^2.
\end{eqnarray*}
Therefore there exists some $r_1 \in (r_0, 2r_0)$ such that
$$ \int_{\partial B(x_0,r_1)} |f|^2d\calH^{m-1} \leq \frac{2}{r_0} \int_{B(0,3r_0)} |f|^2. $$
Moreover one has, according to (\ref{cor_freq2}),
\setlength\arraycolsep{1.4pt}
\begin{eqnarray*}
\int_{B(x_0,r)} |f|^2 & = & \int_0^r \int_{\partial B(x_0,\rho)} |f|^2 \\
& = & \int_0^r H_{x_0}(\rho) d\rho \\
& \leq & \int_0^r \big( \frac{\rho}{r_0} \big)^{m-1} H_{x_0}(r_0) d\rho \\
& = & \frac{r^m}{m} \frac{H_{x_0}(r_0)}{r_0^{m-1}} \\
& \leq & \frac{r^m}{m} \frac{H_{x_0}(r_1)}{r_1^{m-1}} \\
& \leq & \frac{Cr^m}{r_0^m} \int_{B(0,3r_0)} |f|^2.
\end{eqnarray*}
Therefore
\begin{equation}\label{estimate_proof_bounded}
\frac{1}{\calL^m(B(x_0,r))} \int_{B(x_0,r)} |f|^2 \leq \frac{C}{r_0^m} \int_{B(0,3r_0)} |f|^2.
\end{equation}
Whenever $x_0$ is a Lebesgue density point of $|f|^2 \in L_1(U)$, by letting $r$ tend to 0, (\ref{estimate_proof_bounded}) gives the desired estimate.
\end{proof}

\begin{Remark}
If $a$ is a Lebesgue density total branch point of a squeeze and squash stationary map $f \in W^1_2(U, \QQl)$, \emph{i.e.} if there exists $y \in \ell_2$ such that
$$
\lim_{r \to 0} \frac{1}{\calL^m(B(a,r))} \int_{B(a,r)} \calG(f(x), Q\lseg y \rseg)^2 = 0,
$$
then $f$ is continuous at $a$. Indeed, apply Proposition \ref{stationary_bounded} to the squeeze and squash
stationary map $\tau_{Q \lseg y \rseg}(f)$, one has that
$$ \lim_{r \to 0}\sup_{x \in B(a,r)} \calG(f(x), f(a))^2 = \lim_{r \to 0} \frac{C}{r^m} \int_{B(a,3r)} | \tau_{Q \lseg y \rseg}(f) |^2 = 0.
$$
\end{Remark}

Now, for $f \in W^1_2(U, \QQl)$ and $B(a,r) \subset U$, we will denote by $\overline{f}_{a,r}$ a mean of $f$ in the ball $B(a,r)$.

\begin{Proposition}\label{VMO}
If $f \in W^1_2(U, \QQl)$ is squeeze and squash stationary, then
\begin{enumerate}
\item[(1)] for any $a$ in $U$,
$$ \lim_{r\rightarrow 0}\Theta_a(r)=\lim_{r \to 0} \frac{1}{r^{m-2}} \int_{B(a,r)} \lno Df\rno^2 = 0.$$
\item[(2)] in case $U = B(0,1)$, $f$ is $VMO$ in the following sense: there is a function $\omega : [0, 2^{-1}] \to \R_+$ satisfying $\displaystyle\lim_{r \to 0} \omega(r) = 0$ such that for any $a \in B(0,2^{-1})$, $r \in [0, 2^{-1}]$,
$$ \frac{1}{\calL^m(B(a,r))} \int_{B(a,r)} \calG^2(f, \overline{f}_{a,r})\leq \omega(r). $$
\end{enumerate}
\end{Proposition}

\begin{proof}
As in (\ref{monotonic_measure}), we introduce
$$ \mu(A) := \int_A \lno Df \rno^2. $$ We consider two cases. \\
\emph{Case 1.} $\displaystyle\lim_{r \to 0} N_a(r) = \alpha > 0$.

We then use the monotonicity of frequency (Theorem \ref{monotonic_frequency}) and (\ref{cor_freq3}) to infer that, for some $r_0 < \mathrm{dist}(a, \partial U)$ and $r \in (0, r_0]$,
$$
\frac{D_a(r)}{r^{m-2}} \leq \left( \frac{r}{r_0}\right)^{2\alpha} \frac{D_a(r_0)}{r_0^{m-2}},
$$
which converges to 0 as $r \to 0$.\\
\emph{Case 2.} $\displaystyle\lim_{r \to 0} N_a(r) = 0$.

Then, by definition of $N_a(r)$ and by Proposition \ref{stationary_bounded},
$$
\frac{D_a(r)}{r^{m-2}} = N_a(r) \frac{H_a(r)}{r^{m-1}} \leq C N_a(r) \|f\|_{L^\infty(U)},
$$
which converges to 0 as well. This establishes (1).

We now come to (2). The functions $\Theta_a : [0,2^{-1}] \to \R_+$, $a \in B(0,2^{-1})$, are monotone nondecreasing with respect of $r$ by Proposition \ref{energy_nondecreasing}, continuous and satisfy $\Theta_a(0) = 0$ by (1). Let us introduce
\begin{equation}\label{definition_omega}
 \omega(r) := \sup_{a \in B(0, 2^{-1})} \frac{\mu(B(a,r))}{r^{m-2}}.
\end{equation}
By the classical Dini theorem, $\displaystyle\lim_{r \to 0} \omega(r) = 0$. Then (2) follows from Poincar\'e's inequality.
\end{proof}

\begin{Proposition}[Logarithmic decay of normalized energy]
Let $f$ be a squeeze and squash stationary map in $W_2^1(U, \QQl)$, $B(0,3) \subset U$. Then there exist $C$ (depending on $\|f\|_{L^\infty(B)}$ and $\alpha \in (0,1)$) such that for every $a \in B(0,2^{-1})$ and $r \in (0,2^{-1}]$ one has
$$ \frac{1}{\calL^m(B(a,r))} \int_{B(a,r)} \calG^2(f, \overline{f}_{a,r}) \leq C \big(\frac{1}{|\ln r|}\big)^\alpha.$$
In fact,
$$
\frac{1}{r^{m-2}} \int_{B(a,r)} \lno Df \rno^2 \leq C \big(\frac{1}{|\ln r|}\big)^\alpha.
$$
\end{Proposition}

\begin{proof}
We start with $\rho_0 = \frac12$ and define a sequence $\{\rho_j\}_j$ inductively by $\rho_{j+1} = \rho_{j}^2$. Thus $ \rho_j = \rho_0^{2^j}$.

Note that, using Proposition \ref{stationary_bounded}, one has
\begin{equation}\label{log_decay1}
\Theta_a(\rho_{j+1}) =\frac{D_a(\rho_{j+1})}{\rho_{j+1}^{m-2}} = N_a(\rho_{j+1}) \frac{H_a(\rho_{j+1})}{\rho_{j+1}^{m-1}} \leq C  N_a(\rho_{j+1}),
\end{equation}
and by (\ref{cor_freq3})
\begin{equation}\label{log_decay2}
\Theta_a(\rho_{j+1}) = \frac{D_a(\rho_{j+1})}{\rho_{j+1}^{m-2}} \leq \left( \frac{\rho_{j+1}}{\rho_{j}}\right)^{2N_a(\rho_{j})} \frac{D_a(\rho_{j})}{\rho_{j}^{m-2}}
= \left( \frac{\rho_{j+1}}{\rho_{j}}\right)^{2N_a(\rho_{j})} \Theta_a(\rho_{j}).
\end{equation}
First suppose that $\rho_j^{2N_a(\rho_j)} \leq 2^{-1}$. Then by (\ref{log_decay2}),
$$ \Theta_a(\rho_{j+1}) \leq \left( \frac{\rho_{j+1}}{\rho_{j}}\right)^{2N_a(\rho_{j})} \Theta_a(\rho_{j}) = \rho_{j}^{2N_a(\rho_j)} \Theta_a(\rho_j) \leq \frac{1}{2} \Theta_a(\rho_j). $$
On the other hand, if $\rho_j^{2N_a(\rho_j)} > 2^{-1}$, then
$$ 2 N_a(\rho_j) \ln \rho_j > \ln \big( \frac{1}{2} \big).$$
As $\rho_j = 2^{-2^j}$, we infer that $N_a(\rho_j) < 2^{-j-1}$. In that case, (\ref{log_decay1}) yields
$$ \Theta_a(\rho_{j+1}) \leq C  2^{-j-1}.$$
Recalling (\ref{definition_omega}), we prove that
$$ \omega(\rho_{j+1}) \leq \max \big\{C 2^{-j-1}  , \frac{1}{2} \omega(\rho_j) \big\}.$$
It is now standard that for some $\alpha \in (0,1)$,
$$ \omega(\rho) \leq C\big(\frac{1}{|\ln \rho |}\big)^{\alpha}.$$
We conclude the proof by applying Proposition \ref{VMO}(2).
\end{proof}

\bigskip
\noindent{\bf Acknowledgements}.  The second author was supported in part by the Project ANR-12-BS01-0014-01 Geometry.
The third author is partially supported by NSF grants DMS1001115
and DMS 1265574,  NSFC grant 11128102,  and a Simons Fellowship in Mathematics.


\addcontentsline{toc}{section}{Bibliography}
\begin{thebibliography}{99}

\bibitem{Almgren} F. Almgren, {\it Almgren's big regularity paper}.  World Scientific Monograph Series in Mathematics, Vol. {\bf 1}, World
Scientific Publishing Con. Inc., River Edge, NJ, 2000.

\bibitem{DS} C. De Lellis and E. N. Spadaro, {\it Q-valued functions revisited}.
Memoirs Amer. Math. Soc. {\bf 211} (2011), no. 991.

\bibitem{D} C. De Lellis, {Errata to {\it Q-valued functions revisited}}. http://user.math.uzh.ch/delellis.
\bibitem{BDPG} P. Bouafia, T. De Pauw, J. Goblet,
{\it Existence of $p$-harmonic multiple valued maps into a separable Hilbert space}. Preprint (2012).

\bibitem{Lu} S.  Luckhaus, {\it Partial H\"older continuity for minima of certain energies among maps into a Riemannian manifold.}
Indiana Univ. Math. J. {\bf 37} (1988), no. 2, 349-367.

\bibitem{Morrey} C. B. Morrey, Multiple integrals in the calculus of variations. Die Grundlehren der mathematischen Wissenschaften, Band 130 Springer-Verlag New York, Inc., New York 1966 ix+506 pp.

\end{thebibliography}
\end{document}